\documentclass[11pt,reqno]{amsart}

\usepackage[all]{xy}
\usepackage{amssymb}
\usepackage{mathrsfs}
\usepackage{cite}
\usepackage{amsfonts}
\newcommand{\rmv}[1]{}

\newcommand{\RM}{\mathsf {RM}}

\newcommand{\RLM}{\mathsf {RLM}}

\newcommand{\wh}{\widehat}

\newcommand{\MM}{\ensuremath{{\mathscr M}}}
\newcommand{\pv}[1]{\ensuremath{{\bf#1}}}
\newcommand{\ZZ}{\ensuremath{\mathbb Z}}

\newcommand{\ilim}{\varprojlim}

\newcommand{\inv}{^{-1}}
\newcommand{\p}{\varphi}
\newcommand{\pinv}{{\p \inv}}
\newcommand{\J}{\mathrel{{\mathscr J}}} 

\newcommand{\R}{\mathrel{{\mathscr R}}} 
\newcommand{\eL}{\mathrel{{\mathscr L}}} 
\newcommand{\HH}{\mathrel{{\mathscr H}}} 

\newcommand{\ov}[1]{\ensuremath{\overline {#1}}}

\newcommand{\til}[1]{\ensuremath{\widetilde {#1}}}

\newcommand{\NN}{\ensuremath{\mathbb N}}
\newcommand{\malce}{\mathbin{\hbox{$\bigcirc$\rlap{\kern-8.3pt\raise0,50pt\hbox{$\mathtt{m}$}}}}}

\newcommand{\ovFP}[2]{\ensuremath{\widehat{F}_{\overline{\mathbf #1}}(#2)}}

\newcommand{\Thmname}{Theorem}
\newcommand{\Propname}{Proposition}
\newcommand{\Lemmaname}{Lemma}
\newcommand{\Definitionname}{Definition}
\newtheorem{Thm}{\Thmname}[section]
\newtheorem{Prop}[Thm]{\Propname}
\newtheorem{Lemma}[Thm]{\Lemmaname}
{\theoremstyle{definition}
\newtheorem{Def}[Thm]{\Definitionname}}
{\theoremstyle{remark}
\newtheorem{Rmk}[Thm]{Remark}}
\newtheorem{Cor}[Thm]{Corollary}

\newtheorem*{Lemma*}{Lemma}

\numberwithin{equation}{section}

\title{Profinite Groups Associated to Sofic Shifts are Free}

\author{Alfredo Costa\and Benjamin Steinberg}
\address{CMUC, Department of Mathematics\\
  University of Coimbra, 3001-454 Coimbra\\
  Portugal\and School of Mathematics and Statistics\\
Carleton University \\
1125 Colonel By Drive\\
Ottawa, Ontario  K1S 5B6 \\
Canada}
\thanks{
The authors
  acknowledge the support of the research programme
     AutoMathA of ESF.     
  The first author was supported by FCT project
  PTDC/MAT/65481/2006 and FCT post-doctoral grant
  SFRH/BPD/46415/2008.  
  The second author was supported in part by NSERC and the DFG}
\email{amgc@mat.uc.pt\and bsteinbg@math.carleton.ca}
\date{August 4, 2009}


\begin{document}
\begin{abstract}
We show that the maximal subgroup of the free profinite semigroup
associated by Almeida to an irreducible sofic shift is a free
profinite group, generalizing an earlier result of the second author
for the case of the full shift (whose corresponding maximal subgroup
is the maximal subgroup of the minimal ideal).  A corresponding result
is proved for certain relatively free profinite semigroups.
We also establish some other analogies between the
kernel of the free profinite semigroup and the $\J$-class associated to an
irreducible sofic shift.
\end{abstract}
\maketitle

\section{Introduction}
The study of maximal subgroups of free profinite semigroups has
recently received quite a bit of attention in the
literature~\cite{projective,minimalideal,primitivesub,Almeida:2005bshort,Alfredo2,AlmeidaVolkov2}.
Almeida discovered how to associate to each irreducible symbolic
dynamical system a maximal subgroup of a free profinite
semigroup~\cite{Almeida:2003b,Almeida:2005bshort,Almeida:2005USU}.
In~\cite{Almeida:2003b,Almeida:2005bshort} he announced
that this subgroup is invariant under conjugacy of dynamical systems,
but flaws were detected in the arguments sketched in~\cite{Almeida:2003b}.
The first author used a different approach to successfully prove
the conjugacy invariance of the maximal group~\cite{Alfredo2}.
The resolution of the flaws in~\cite{Almeida:2003b} led to the paper~\cite{Almeida&Costa:2009},
making possible for its authors to produce a proof
according to the original approach of Almeida; such a proof appears in~\cite{Costa:2007}.

In~\cite{primitivesub}, Almeida studied the case of minimal systems
associated to primitive substitutions and under certain hypotheses,
the corresponding maximal subgroup was shown to be a free profinite
group.  An example of a non-free maximal subgroup, with rank two,
associated to a primitive substitution was also obtained
in~\cite{primitivesub}.
It has since been proved by Almeida and the first author that the
maximal subgroup associated to the Thue-Morse dynamical system is
not free profinite and has rank three~\cite{Almeida&Costa:2009b}.
In particular, a question of Margolis from 1997 as to whether all maximal subgroups of a free profinite semigroup are free was answered in the negative.  Margolis also asked at this time whether all maximal subgroups were projective profinite groups and whether the maximal subgroup of the minimal ideal was a free profinite group.  The first question was answered in the positive by Rhodes and the second author~\cite{projective}, whereas the second was answered positively by the second author~\cite{minimalideal}.

The maximal subgroup of the minimal ideal is the maximal subgroup of
the free profinite semigroup on $X$
associated to the full shift $X^{\mathbb Z}$, which is an
example of an irreducible sofic shift.
It is then natural to ask whether the maximal subgroup associated to any irreducible sofic shift is free.  A sofic shift is minimal if and only if it is periodic.  In this case, Almeida and Volkov established early on that the corresponding maximal subgroup is free procyclic~\cite{AlmeidaVolkov2}.  In this paper we show that the maximal subgroup associated to a non-minimal irreducible sofic shift is a free profinite group of countable rank, thereby generalizing the result for the minimal ideal~\cite{minimalideal}.  An interesting feature of the proof is the crucial role played by the invariance of the subgroup under conjugacy of dynamical systems.   A consequence of our results is that there is a dense set of idempotents whose corresponding maximal subgroups are free.  Actually, we prove a stronger result that applies to certain relatively free profinite groups; the precise statement is left to the body of the article.

Several intermediate results established in the paper are likely to be of interest to researchers in symbolic dynamics and finite semigroup theory.  For instance, we characterize the syntactic semigroups of irreducible sofic shifts as precisely the generalized group mapping semigroups~\cite{Arbib,qtheor} with aperiodic $0$-minimal ideal.  Fischer covers can then be interpreted as the corresponding Sch\"utzenberger graphs.

The paper is organized as follows.  The first section consists of
preliminaries about semigroups and languages.  This is followed by a
section on sofic shifts.  The necessary background on sofic shifts and
their relationship with free profinite semigroups is given, as well as
several new results.  The fourth section reviews the wreath product of
partial transformation semigroups and establishes our notational
conventions for iterated wreath products.
The fifth section states our main result and proves it modulo a
technical lemma.
The following section proves the technical lemma, which is based on
the argument of~\cite{minimalideal}.
Other results from~\cite{AlmeidaVolkov2} about the minimal ideal of the free profinite semigroup are generalized
to the minimal $\J$-class associated to arbitrary irreducible sofic subshifts. Namely, in the seventh section the existence of computable idempotents in such $\J$-classes is established; in
the last section, using the notion of entropy,
we obtain a characterization of this $\J$-class that is used to prove
that, roughly speaking,
we can not reach its elements
starting with strict factors
and using only iterations of certain endomorphisms and compositions of
implicit operations with low arity.

\section{Semigroups and languages}
Throughout this paper we shall use basic notions from semigroup theory that can be found in standard texts~\cite{CP,Eilenberg,Arbib,qtheor,Almeida:book,grillet}. In particular, recall that Green's (equivalence) relation $\J$ is defined on a semigroup $S$ by putting $s\J t$ if $s$ and $t$ generate the same two-sided principal ideal.  The $\J$-class of an element $s$ is denoted $J_s$.  We use the notation $s\leq_{\J} t$ to indicate that the two-sided ideal generated by $s$ is contained in that generated by $t$.  This is a preorder descending to an order $S/{\J}$. Sometimes we use $s\geq_{\J} J$ as shorthand for $J_s\geq_{\J} J$.  Similarly defined are the $\R$- and $\eL$-relations, where right (respectively, left) ideals replace two-sided ideals.  Analogous notation is used for $\eL$- and $\R$-classes.  The intersection of $\R$ and $\eL$ is denoted $\HH$.  The $\HH$-class of an idempotent $e$ of a semigroup $S$ is a group, called the \emph{maximal subgroup} at $e$.  It can alte
 rnatively be defined as the group of units of the monoid $eSe$.

By a \emph{compact semigroup} $S$, we mean a non-empty semigroup with a compact Hausdorff topology such that multiplication is jointly continuous.  A \emph{profinite semigroup} is a projective limit of finite semigroups, or equivalently a compact totally disconnected semigroup. Basic information about compact and profinite semigroups can be found in~\cite[Chapter 3]{qtheor}. A $\J$-class of a compact semigroup $S$ is \emph{regular} if it contains an idempotent, or equivalent all its elements are von Neumann regular.  If $S$ is a compact semigroup and $e,f\in S$ are $\J$-equivalent idempotents, then $G_e\cong G_f$ and so each regular $\J$-class has a unique maximal subgroup up to isomorphism of topological groups.  Every compact semigroup has a unique minimal ideal, which is necessarily principal and hence closed.   The minimal ideal is always a regular $\J$-class.

If $X$ is a set, the free semigroup on $X$ is denoted $X^+$; the corresponding free monoid is $X^*$; the respective profinite completions are denoted $\wh{X^+}$ and $\wh{X^*}$.  A subset $L\subseteq X^+$ is often called a \emph{language}.  A language is \emph{rational} if it can be recognized by a finite state automaton.  Equivalently, $L$ is rational if there is a finite semigroup $S$ and a homomorphism $\p\colon X^+\to S$ so that $L=\pinv\p(L)$.  The category of onto morphisms recognizing $L\subseteq X^+$ has a terminal object $\lambda\colon X^+\to S_L$ called the \emph{syntactic morphism} of $L$.  The semigroup $S_L$ is called the \emph{syntactic semigroup} of $L$ and is the quotient of $X^+$ by the congruence that puts $x\equiv y$ if, for all $u,v\in X^*$, one has $uxv\in L\iff uyv\in L$.  See~\cite{EilenbergA} for details.

\section{Sofic shifts}

\subsection{Definitions and notation}
A good reference for the notions that we shall use here from symbolic dynamics is~\cite{MarcusandLind}.
Let $X^{\ZZ}$ be the set of biinfinite sequences of letters of $X$ indexed by $\ZZ$.
The \emph{shift} on $X^{\ZZ}$ is the bijective
map $\sigma_X$ (or just $\sigma$) from
$X^{\ZZ}$ to $X^{\ZZ}$ defined by
 $\sigma_X((x_i)_{i\in \ZZ})=(x_{i+1})_{i\in \ZZ}$.
 The \emph{orbit} of $x\in X^{\mathbb Z}$ is the set
 $\{\sigma^k(x)\mid k\in\ZZ\}$.
We endow $X^{\ZZ}$ with the product topology with respect to the
discrete topology of $X$.
A \emph{symbolic dynamical system}
is a non-empty closed subset $\mathscr X$ of some $X^{\ZZ}$ invariant under $\sigma$.
Symbolic dynamical systems are also called
\emph{shift spaces} or \emph{subshifts}.

Two subshifts $\mathscr X\subseteq X^\mathbb{Z}$ and $\mathscr Y\subseteq
Y^\mathbb{Z}$ are \emph{topologically conjugate} if there
is a homeomorphism $\varphi\colon \mathscr X\to \mathscr Y$ commuting with shift:
$\varphi\circ\sigma_X=\sigma_Y\circ\varphi$. Such a homeomorphism is
also called a \emph{topological conjugacy}. Since we will consider no
other form of conjugacy, we drop the reference to its topological
nature.

Let $x=(x_i)_{i\in\ZZ}$
be an element of $X^{\mathbb Z}$.
We may represent it by
\begin{equation*}
  x=\ldots x_{-3}x_{-2}x_{-1}.x_{0}x_{1}x_{2}\ldots
\end{equation*}
where the central dot indicates that the $0$ coordinate of $x$
is the letter at its immediate right.

  By a \emph{factor of $(x_i)_{i\in\ZZ}$} we mean a word
  $x_ix_{i+1}\cdots x_{i+n-1}x_{i+n}$
  (briefly denoted by $x_{[i,i+n]}$), where $i\in\ZZ$ and
  $n\geq 0$.
     If $\mathscr X$ is a subset of $X^{\ZZ}$ then we denote by
  $L(\mathscr X)$ the set
  of factors of elements of $\mathscr X$.
  A subset $L$ of a semigroup $S$ is said to be \emph{factorial} if it is closed under
  taking factors,
  and it is \emph{prolongable} if
   for every element $u$ of $L$
   there are elements $a,b\in S$ such that $aub\in L$.
   It is easy to prove that the correspondence
   $\mathscr X\mapsto L(\mathscr X)$ is an order isomorphism between
   the lattice of subshifts of
   $X^{\ZZ}$ and the lattice of non-empty, factorial, prolongable languages
   in $X^+$~\cite[Proposition 1.3.4]{MarcusandLind}.

A subset $L$ of a semigroup $S$ is said to be \emph{irreducible} if,
for all $u,v\in L$, there exists $w\in S$ so  that $uwv\in L$.  Notice
that, for factorial sets, irreducibility implies prolongability since if $u\in L$, then $uwuw'u\in L$ for some $w,w'\in S$ and hence $wuw'\in L$ since $L$ is factorial.
A shift $\mathscr X$ is said to be \emph{irreducible} if $L(\mathscr X)$ is an irreducible subset of $X^+$; this is equivalent to saying that $\mathscr X$ has a dense forward orbit~\cite{MarcusandLind}.  One says that $\mathscr X$ is \emph{minimal} if it contains no proper subshift.  Minimal shifts are irreducible~\cite{MarcusandLind}.

\subsection{Sofic shifts}
It is natural to consider those shifts whose associated language is rational.

\begin{Def}[Sofic shift]
A shift $\mathscr X$ is \emph{sofic} if $L(\mathscr X)$ is rational.
\end{Def}
A shift of \emph{finite type} is a shift
$\mathscr X$ such that $L(\mathscr X)=X^+\setminus X^{\ast}FX^{\ast}$ for some
finite set $F$. Therefore finite type shifts are sofic.  Sofic shifts are exactly the quotients (or factors) of shifts of finite type.

Recall that a shift is called \emph{periodic} if it is the (finite) orbit of a word of the form $x_u=\ldots uu.uuu\ldots$ with $u\in X^+$.  The following is well known.

\begin{Lemma}\label{periodic}
A sofic shift is minimal if and only if it is periodic.
\end{Lemma}
\begin{proof}
It is easy to see that every periodic shift is minimal.  Suppose conversely that $\mathscr X$ is a minimal sofic shift.  Since $L(\mathscr X)$ is an infinite factorial rational language, by the Pumping Lemma we can find a non-empty word $u\in X^+$ so that $u^n\in L(\mathscr X)$ for all $n\geq 0$.  The orbit of $x_u=\ldots uuu.uuuu\ldots$ is then a subshift of $\mathscr X$, which is hence periodic by minimality.
\end{proof}

An example of an irreducible sofic shift is the \emph{full shift}
$X^{\mathbb Z}$.

\subsection{Coding}

Let $N$ be a positive integer.
Let $Y_{N}$ be the alphabet $X^{N}$.
We shall use the notation $[u]$ for a word $u\in X^{N}$ when
we want to consider it as a letter of $Y_N$.

Let $N>0$ and define $\beta_{N}\colon X^{\mathbb Z}\to Y_{N}^{\ZZ}$
by
\begin{equation*}
\beta_{N}((x_i)_{i\in\ZZ})=([x_{[i,i+N-1]}])_{i\in\ZZ}.
\end{equation*}
For example,
\begin{align*}
  \beta_3(\cdots x_{-3}x_{-2}x_{-1}.&x_{0}x_{1}x_{2}\cdots) =\\
  &\cdots [x_{-2}x_{-1}x_{0}][x_{-1}x_{0}x_{1}].[x_{0}x_{1}x_{2}][x_{1}x_{2}x_{3}]
  [x_{2}x_{3}x_{4}]\cdots
\end{align*}
Given a subset $\mathscr X$ of $X^{\mathbb Z}$, denote by $\mathscr X^{[N]}$
the set $\beta_{N}(\mathscr X)$.
The following is \cite[Example 1.5.10]{MarcusandLind}.

\begin{Lemma}\label{l:exercise}
If $\mathscr X$ is a subshift of $X^{\mathbb Z}$, then $\mathscr
X^{[N]}$ is a subshift of
${(Y_N)}^{\mathbb Z}$
  and the map $\beta_{N}\colon \mathscr X\to\mathscr X^{[N]}$ is a conjugacy.
\end{Lemma}

Given a word $u$, denote by $\mathsf{alph}(u)$ the set of letters occurring in $u$.  We extend this notation to shifts $\mathscr X\subseteq X^{\ZZ}$ by putting $\mathsf{alph}(\mathscr X) = X\cap L(\mathscr X)$.

\begin{Lemma}\label{reducedtonicecase}
Let $\mathscr X$ be a sofic shift, which is not minimal.
Then there is a conjugate shift $\mathscr Y$ of $\mathscr X$
for which there is a non-empty word $z\in L(\mathscr Y)$ such that $z^+\subseteq L(\mathscr Y)$ and $\mathsf{alph}(z)\subsetneq \mathsf{alph}(\mathscr Y)$.
\end{Lemma}
\begin{proof}
Suppose that $\mathscr X\subseteq X^{\mathbb Z}$ is a sofic shift, but not minimal.  By the Pumping Lemma, there is a non-empty word $w$ so that $w^+\subseteq L(\mathscr X)$.
The orbit of $x_w=\ldots www.wwww\ldots$ is periodic and thus a
minimal
shift contained in $\mathscr X$.
Since $\mathscr X$ is not minimal,
the set $L(\mathscr X)\setminus L(x_w)$
contains some element $v$.  Because $x_w=x_{w^m}$ for all $m>0$, we may as well suppose that $|w|\geq |v|$.
Using that $L(\mathscr X)$ is prolongable, we may in fact assume $|w|=|v|$.  The fact that $v\notin L(x_w)$ then translates into saying that $v$ is not a cyclic conjugate of $w$. Set $N=|w|=|v|$.

Recall that $\mathscr X^{[N]}$ is conjugate to $\mathscr X$ by
Lemma~\ref{l:exercise}.  Suppose that $w=a_1\cdots a_{N}$ with the
$a_i\in X$.  Then setting \[z = [a_1\cdots a_{N}][a_2\cdots
a_{N}a_1]\cdots [a_Na_1a_2\cdots a_{N-1}]\] we have $z^+\subseteq
L(\mathscr X^{[N]})$ (as $\beta_{N}(x_w)\in \mathscr X^{[N]}$). But
since $v$ is not a cyclic conjugate of $w$, we have $[v]\notin
\mathsf{alph}(z)$. On the other hand, since $v$ has length $N$ and is a factor of some $x\in \mathscr X$, the letter $[v]$ is
a factor of $\beta_{N}(x)\in \mathscr X^{[N]}$.  Thus $[v]\in \mathsf{alph}(\mathscr X^{[N]})$.
\end{proof}

\subsection{Regular $\J$-classes of compact semigroups}
 The goal of this subsection is to establish a bijection between regular $\J$-class of a compact semigroup $S$ and non-empty, factorial, irreducible subsets (\emph{FI-subsets}) of $S$ which are closed topologically.  This sets the stage for the connection with symbolic dynamics. We begin with a lemma on inverse images of such sets.  We remark that a factorial set is a union of $\J$-classes.
If $S$ is a compact semigroup, $S^1$ denotes $S$ with a functorially adjoined identity element $1$, which is topologically an isolated point.

\begin{Lemma}\label{pullback}
Let $\p\colon S\to T$ be a homomorphism of semigroups and suppose $\emptyset\neq L\subseteq T$. Then:
\begin{enumerate}
\item If $L$ is factorial, $\pinv (L)$ is factorial;
\item If $\p$ is surjective and $L$ is irreducible, then $\emptyset\neq \p\inv(L)$ is irreducible.
\end{enumerate}
\end{Lemma}
\begin{proof}
For (1), if $w$ is a factor of $\pinv (L)$, evidentally $\p(w)$ is a factor of $L$ and so $\p(w)\in L$, whence $w\in \pinv (L)$.  Thus $\pinv(L)$ is factorial.   On the other hand, suppose $\p$ is onto and $L$ is irreducible.  Assume $u,v\in \pinv (L)$ and choose $w\in T$ so that $\p(u)w\p(v)\in L$.  Then if $\til w$ is a preimage of $w$, one has $u\til wv\in \pinv (L)$.  This completes the proof.
\end{proof}

The following proposition characterizes the closed FI-subsets of a compact semigroup.  It combines the fundamental idea of Rhodes for lifting regular $\J$-classes~\cite{Arbib,qtheor} and an idea of Almeida on irreducible shifts~\cite{Almeida:2005bshort}.  If $A$ is a subset of a semigroup, we denote by $\mathrm{Fact}(A)$ the set of all factors of $A$.

\begin{Prop}\label{JclassesasFI}
Let $S$ be a compact semigroup.
\begin{enumerate}
\item  If $A$ is a closed, non-empty, factorial, irreducible subset, then there is a unique minimal $\J$-class $J(A)$, called the \textbf{apex} of $A$, such that $J(A)\subseteq A$.  Moreover, $J$ is regular and $A=\mathrm{Fact}(J(A))$.
\item If $J$ is a regular $\J$-class, then $\mathrm{Fact}(J)$ is a closed, non-empty, factorial, irreducible subset with apex $J$.
\end{enumerate}
Consequently, regular $\J$-classes of $S$ are in bijection with closed FI-subsets.
\end{Prop}
\begin{proof}
To prove (1), first observe that $A$ is a union of $\J$-classes. We next show that every element $s\in A$ is $\J$-above a minimal $\J$-class of $A$.  Let $\mathscr C$ be the set of all closed ideals of $S$ intersecting $A$, which are contained in the principal ideal generated by $s$.  Then $\mathscr C$ is non-empty (as it contains the principal ideal generated by $s$) and by compactness the intersection of any descending chain of elements of $\mathscr C$ meets $A$ in a non-empty subset, and hence belongs to $\mathscr C$ .  Thus $\mathscr C$ has a minimal element $I$ by Zorn's Lemma.  If $a\in A\cap I$, then the principal ideal generated by $a$ intersects $A$ and is contained in $I$.  Thus $I$ is a principal ideal by minimality and so $A$ contains a minimal $\J$-class, which is $\J$-below $s$.

Suppose $J_1,J_2$ are minimal $\J$-classes of $A$ (perhaps equal).  Let $u\in J_1$ and $v\in J_2$.  Then by irreducibility, there exists $w\in S$ so that $uwv\in A$.  Clearly, $uwv\leq_{\J} J_1,J_2$.  We conclude by minimality that $J_1=J_2$ and $J_1^2\cap J_1$ is non-empty.  Thus $J_1$ is regular and unique.  From now on we denote it $J(A)$.  Clearly, $\mathrm{Fact}(J(A))\subseteq A$.  Conversely, if $a\in A$, then we know $a$ is $\J$-above a minimal $\J$-class of $A$, which must be $J(A)$ by uniqueness.  This establishes (1).

For (2), first note that if $\{x_{\alpha}\}$ is a net in
$\mathrm{Fact}(J)$ converging to $x\in S$, then we can find, for each
$\alpha$, elements $u_{\alpha},v_{\alpha}\in S^1$ so that
$u_{\alpha}x_{\alpha}v_{\alpha}\in J$.  By passing to a subnet, we may
assume that $u_{\alpha}\to u$ and $v_{\alpha}\to v$ and hence
$u_{\alpha}x_{\alpha}v_{\alpha}\to uxv$.  Then since $\J$-classes of a
compact semigroup are closed~\cite[Proposition 3.1.9]{qtheor} it
follows that $uxv\in J$ and so $x\in \mathrm{Fact}(J)$.  We conclude
$\mathrm{Fact}(J)$ is closed.  It is clearly factorial.  Suppose
$u,v\in \mathrm{Fact}(J)$.  Then we can find $x,y,x',y'\in S^1$ so
that $xuy,x'vy'\in J$.  Since $J$ is regular, we can find an element
$w\in J$ so that $xuywx'vy'\in J$.  Indeed, there is an $\R$-class $R$
of $J$ so that $L_{xuy}\cap R$ contains an idempotent and an
$\eL$-class $L$ of $J$ with $R_{x'vy'}\cap L$ containing an
idempotent. We can then take $w$ to be any element of the $\HH$-class
$R\cap L$.  Hence $uywx'v\in \mathrm{Fact}(J)$ and $ywx'\in S$.  Thus $\mathrm{Fact}(J)$ is irreducible.  Evidentally, $J$ is minimal in $\mathrm{Fact}(J)$ and hence is the apex of $\mathrm{Fact}(J)$.  This completes the proof.
\end{proof}

As a corollary,
we deduce a result on lifting regular $\J$-classes for compact semigroups~\cite[Lemma 3.1.14]{qtheor}.
The analogue for finite semigroups is well known~\cite{qtheor,Arbib}.  If $A$ is a subset of a semigroup $S$, then $E(A)$ denotes the idempotent elements of $A$.

\begin{Lemma}\label{liftclasses}
Let $\p\colon S\to T$ be a continuous surjective homomorphism of compact semigroups and let $J$ be a regular $\J$-class of $T$.  Then:
\begin{enumerate}
\item There is a unique minimal $\J$-class $J'$ of $S$ so that $\p(J')\subseteq J$, which moreover is the apex of $\pinv(\mathrm{Fact}(J))$;
\item The $\J$-class $J'$ is regular and $\p(J')=J$;
\item Each $\R$-class, $\eL$-class and $\HH$-class of $J'$ maps onto a corresponding class of $J$;
\item $\p(E(J'))=E(J)$.
\end{enumerate}
In particular, each maximal subgroup of $J'$ maps homomorphically onto a maximal subgroup of $J$ and each maximal subgroup of $J$ is the image of a maximal subgroup of $J'$.
\end{Lemma}
\begin{proof}
By Proposition~\ref{JclassesasFI}, the set $\mathrm{Fact}(J)$ is a closed FI-subset.  Hence, by Lemma~\ref{pullback}, $\pinv(\mathrm{Fact}(J))$ is a closed FI-subset and thus contains by Proposition~\ref{JclassesasFI}  a unique minimal $\J$-class $J'$, which moreover is regular, and $\pinv(\mathrm{Fact}(J)) = \mathrm{Fact}(J')$.  Suppose $x\in J$ where $x=\p(s)$ with $s\in \mathrm{Fact}(J')$.  Then we can find $u,v\in S^1$ so that $usv\in J'$.  Then $\p(usv)\in J$.  Since $\p$ is onto, we can find $a,b\in S^1$ so that $\p(ausvb)=x$.  Hence $ausvb\in J'$ by minimality.  We conclude $\p(J')\supseteq J$.  But $\p(J')$ must be contained in a $\J$-class of $T$ so $\p(J')=J$.

Suppose $R$ is an $\R$-class of $J'$ and that $t\R \p(r)$ with $r\in R$.  Then we can find $s\in S$ so that $\p(r)\p(s)=t$.  Then $rs\leq_{\J} r$ and so $rs\in J'$ by minimality, whence $rs\R r$ by stability of compact semigroups~\cite[Chapter 3]{qtheor}.  Thus $t\in \p(R)$.

Next suppose $G'$ is a maximal
subgroup of $J'$ with identity $e$.  Let $G$ be the $\HH$-class of
$J$ containing the image of $G'$.  If $g\in G$ and $s\in S$ is any
preimage of $g$, then $\p(ese) = g$ and so, by minimality, $ese\in
J'$. Stability shows in fact $ese\in G'$.  Thus $\p(G')= G$.  Now let
$H'$ be an $\HH$-class $\eL$-equivalent to $G'$ and let $H$ be the
$\HH$-class of $J$ into which $H'$ maps.  Fix $h\in H'$; so $\p(h) \in
H$. Green's Lemma  implies that \[H = \p(h) G = \p(h) \p(G') = \p(hG') =
\p(H').\] Because every $\HH$-class in a regular $\J$-class is
$\eL$-equivalent to a maximal subgroup, this completes the proof of (3).

Finally, to prove (4) suppose that $e\in E(J)$ and let $A=\pinv(e)\cap J'\neq\emptyset$.    Since $\J$-classes in a compact semigroup are closed~\cite[Proposition 3.1.9]{qtheor}, the set $A$ is closed.  But if $s,t\in A$, then $\p(st)=e$ and $st\leq_{\J} s$.  Thus by minimality $st\in J'$ and so $st\in A$.  Thus $A$ is compact semigroup and hence contains an idempotent~\cite[Corollary 3.1.2]{qtheor}.
\end{proof}

If $\pv V$ is a variety of finite semigroups, then $\wh{F}_{\pv V}(X)$
denotes the free pro-$\pv V$ semigroup on
$X$~\cite{Almeida:book,qtheor}.

\begin{Lemma}\label{inherited-properties} 
Let $L$ be a subset of $X^+$.
Let $\iota\colon X^+\to \wh F_{\pv V}(X)$ be the canonical morphism.
If $L$ is irreducible then the subset $\ov{\iota(L)}$ of $\wh F_{\pv V}(X)$ is irreducible.
Moreover, if $L$ is factorial and $\pv V$ contains the syntactic
semigroup of $L$, 
then $\ov{\iota(L)}$ is factorial.
\end{Lemma}

\begin{proof}
Suppose $u,v\in \ov{\iota(L)}$. Then
$u=\lim \iota(u_n)$
and
$v=\lim \iota(v_n)$
for some sequences
$\{u_n\}$, $\{v_n\}$ of elements of $L$.
For each $n$, there is $w_n\in X^+$ such that
$u_nw_nv_n\in L$.
It follows
that $uwv\in \ov{\iota(L)}$ for some accumulation point $w$
of $\{\iota(w_n)\}$.
This completes
the proof of irreducibility.

Suppose $\pv V$ contains the syntactic semigroup of
$L$.
Then $\ov{\iota(L)}$ is open
and $\iota^{-1}\Bigl(\ov{\iota(L)}\Bigr)=L$~\cite{Almeida:book}.
Let $u\in\ov{\iota(L)}$.
Take a factorization $u=xwy$,
where $x$ and $y$ are allowed to be $1$.
Then $x=\lim\iota(x_n)$,
$w=\lim\iota(w_n)$,
and $y=\lim\iota(y_n)$,
for some sequences
$\{x_n\}$, $\{w_n\}$, $\{y_n\}$ of elements of $X^*$ (where $\iota(1)=1$).
Since $\ov{\iota(L)}$ is open,
for all sufficiently large $n$, we have
$x_nw_ny_n\in \iota^{-1}\Bigl(\ov{\iota(L)}\Bigr)=L$.
Since $L$ is factorial, this implies that
$w\in \ov{\iota(L)}$,
thus $\ov{\iota(L)}$ is factorial.
\end{proof}

If $\pv V$ contains the variety of finite nilpotent semigroups, then
$\iota$ is an embedding of $X^+$ in $\wh F_{\pv V}(X)$, and
thus we consider $\iota$ to be an inclusion map.
There are sofic shifts $\mathscr X$ such that
$\ov{L(\mathscr X)}$ is not factorial in
$\wh{F}_{\pv {LSl}}(X)$~\cite[Proposition 3.3]{Costa:2007},
where $\pv {LSl}$ is the variety of finite semigroups whose
local submonoids are semilattices.

From Proposition~\ref{JclassesasFI} and
the first part of Lemma~\ref{inherited-properties}
we immediately deduce the following
result of Almeida announced in~\cite{Almeida:2005bshort} and proved
in~\cite{Almeida:2005USU}.

\begin{Prop}\label{Jclassofsoficshift}
Let $\mathscr X\subseteq X^{\mathbb Z}$ be an irreducible shift.
Let $\pv V$ be a variety of finite semigroups such
that $\ov{\iota(L(\mathscr X))}$ is factorial,
where $\iota\colon X^+\to \wh F_{\pv V}(X)$ is the canonical morphism.
Then there is a unique minimal $\J$-class, denoted $J(\mathscr X)$, of
$\wh F_{\pv V}(X)$ intersecting $\ov{\iota(L(\mathscr X))}$.
Moreover, $J(\mathscr X)$ is regular and is $\J$-below each element of $\ov{\iota(L(\mathscr X))}$.
\end{Prop}

It was proved in~\cite{Almeida&Costa:2009} that if
$\pv V$ is a variety of finite semigroups such that
$\pv V=\pv A\malce \pv V$
then $\ov{L}$ is a factorial subset of $\wh F_{\pv V}(X)$
whenever $L$ is a factorial subset of $X^+$.
Hence, thanks also to
Lemma~\ref{inherited-properties},
one can define the
$\J$-class $J(\mathscr X)$
in Proposition~\ref{Jclassofsoficshift}
whenever at least one of the
following conditions holds:
\begin{enumerate}
\item $\pv V=\pv A\malce \pv V$;
\item \pv V contains the syntactic semigroup of $L(\mathscr X)$.
\end{enumerate}

The maximal subgroup of $J(\mathscr X)$ is called the \emph{profinite group associated to the irreducible shift $\mathscr X$};  it is known to be a conjugacy invariant of $\mathscr X$ if $\pv V=\pv V\ast \pv D$, where $\pv D$
  is the variety of finite semigroups whose idempotents are 
  right zeroes, and $\pv V$ contains the two-element finite semilattice~\cite{Alfredo2}.

It is easy to see 
that if $\pv H$ is a variety of finite groups
then the variety $\ov{\pv H}$ of finite semigroups whose subgroups are in
$\pv H$ is an example of a variety of finite semigroups
satisfying the equations
$\pv V=\pv A\malce \pv V$
and $\pv V=\pv V\ast \pv D$~\cite{Eilenberg,qtheor}.

\subsection{Sofic shifts and generalized group mapping semigroups}
In this section we establish a connection between irreducible sofic shifts and generalized group mappings semigroups with aperiodic $0$-minimal ideals, which we term $\mathsf{AGGM}$-semigroups to be consistent with the notation of~\cite{qtheor}.

\begin{Def}[$\mathsf{AGGM}$-semigroup]
A finite semigroup $S$ is called \emph{generalized group mapping} if it acts faithfully on the left and right of its minimal ideal or it has a $0$-minimal ideal on which it acts faithfully on both the left and right.  The ideal in question, called the \emph{distinguished ideal}, is unique and regular~\cite[Proposition 4.6.22]{qtheor}.  If the ideal is aperiodic, then we call $S$ an \emph{$\mathsf{AGGM}$-semigroup}.
\end{Def}

Generalized group mapping semigroups were introduced by Krohn and Rhodes in their work on the complexity of finite semigroups; see~\cite{KRannals,folleyR,folleyT,Arbib,qtheor}.

Notice that an $\mathsf{AGGM}$-semigroup is either trivial or contains a $0$ and the distinguished ideal is $0$-minimal. If $S$ is an $\mathsf{AGGM}$-semigroup, then by the \emph{distinguished $\J$-class} of $S$, we mean the unique $\J$-class if $S$ is trivial and otherwise we mean $I\setminus\{0\}$ where $I$ is the distinguished ideal.
It follows from~\cite[Proposition 4.6.37]{qtheor} that $S$ is an $\mathsf{AGGM}$-semigroup if and only if there is a regular $\J$-class $J$ of $S$ with the following property: for all $s,t\in S$, one has $s=t$ if and only if, for all $x,y\in J$, \[xsy\in J\iff xty\in J.\]  Moreover, in this case $J$ is the distinguished $\J$-class.  From now on, if $S$ is an $X$-generated profinite semigroup, then $[w]_S$ will denote the image of a word $w\in X^+$ in $S$.

\begin{Thm}\label{AGGM}
Let $S$ be a finite $X$-generated semigroup.  Then $S$ is the syntactic semigroup of $L(\mathscr X)$ for an irreducible sofic shift  $\mathscr X\subseteq X^{\mathbb Z}$ if and only if it is an $\mathsf{AGGM}$-semigroup.
\end{Thm}
\begin{proof}
Suppose first that $\mathscr X$ is a sofic shift and $S$ is the syntactic semigroup of $L(\mathscr X)$.  If $\mathscr X=X^{\mathbb Z}$, then $S$ is trivial and there is nothing to prove.  So assume $\mathscr X$ is a proper shift.  Then $S$ has a $0$ and $L(\mathscr X)$ is the full inverse image of $S\setminus \{0\}$ since $S$ is the syntactic semigroup of a coideal in $X^+$.  Let $J$ be a minimal non-zero $\J$-class of $S$.  We first claim that $J$ is regular.  Indeed, if $[u]_S\in J$, then by irreducibility we can find $w\in X^+$ so that $uwu\in L(\mathscr X)$. Then $0\neq [uwu]_S\leq_{\J} [u]_S$ and so by minimality $[uwu]_S\in J$.  Hence $[wu]_S\in J$ and so $J^2\neq 0$.  Thus $J$ is regular.  Now suppose that $[s]_S,[t]_S$ are such that, for all $x,y\in J$, one has $x[s]_Sy\in J\iff x[t]_Sy\in J$.  Let $u,v\in X^+$ and assume that $usv\in L(\mathscr X)$.  Let $[z]\in J$.  By irreducibility, we can find $w_1,w_2\in X^+$ so that $zw_1usvw_2z\in L(\mathscr X)$.  Set $x=[zw_1
 u]_S$ and $y=[vw_2z]_S$.  Then $x,y\in J$ and $x[s]_Sy\in J$.  Thus $x[t]_Sy\in J$ and hence $zw_1utvw_2z\in L(\mathscr X)$.  It follows $utv\in L(\mathscr X)$.  A symmetric argument shows that $utv\in L(\mathscr X)$ implies $usv\in L(\mathscr X)$.  Thus $[s]_S=[t]_S$.  This establishes that $S$ is an $\mathsf{AGGM}$-semigroup.

Conversely, assume $S$ is an $\mathsf{AGGM}$-semigroup.  Let $J$ be
the distinguished $\J$-class of $S$ and let $\pi\colon X^+\to S$ be
the canonical surjection.  Then $\mathrm{Fact}(J)=S\setminus \{0\}$ is a non-empty, factorial, irreducible subset of $S$ by Proposition~\ref{JclassesasFI}.  Thus $\pi\inv(S\setminus \{0\})$ is a non-empty, factorial, irreducible rational subset of $X^+$ (by Lemma~\ref{pullback}) and hence of the form $L(\mathscr X)$ for an irreducible sofic shift $\mathscr X\subseteq X^{\ZZ}$. It remains to show that $S$ is the syntactic semigroup of $\pi\inv(S\setminus \{0\})$.  To prove this, it suffices to verify that, given $m,n\in S$ such that $rms\neq 0\iff rns\neq 0$ for all $r,s\in S^1$, one has $m=n$.  By the remark before the theorem, it suffices to  prove that, for all $x,y\in J$, we have $xmy\in J\iff xny\in J$.  But $xmy\in J$ if and only if $xmy\neq 0$, if and only if $xny\neq 0$, if and only if $xny\in J$.  This completes the proof of the theorem.
\end{proof}

\begin{Rmk}
The Fischer cover~\cite{Fischercover,Beauquier} of an irreducible sofic shift $\mathscr X$ is nothing more than the Sch\"utzenberger graph associated to the faithful right action of the syntactic semigroup of $L(\mathscr X)$ on an $\R$-class of its distinguished $\J$-class.
\end{Rmk}

The $\J$-classes corresponding to irreducible sofic shifts admit the following topological characterization.

\begin{Prop}\label{classifyJshift}
Let $\pv V$ be a variety of finite semigroups.  Then a regular $\J$-class $J$ of $\wh{F}_{\pv V}(X)$ (for a finite set $X$) is of the form $J(\mathscr X)$ for an irreducible sofic shift $\mathscr X$ with $L(\mathscr X)$ a $\pv V$-recognizable set if and only if $\mathrm{Fact}(J)$ is clopen.
\end{Prop}
\begin{proof}
 Let $\iota\colon X^+\to \wh{F}_{\pv V}(X)$
denote the canonical morphism. 
Suppose first that  $J=J(\mathscr X)$ for an irreducible sofic shift $\mathscr X$ with $L(\mathscr X)$ a $\pv V$-recognizable set.  Let $\lambda\colon \wh{F}_{\pv V}(X)\to S_{\mathscr X}$ be the continuous homomorphism induced by the syntactic morphism for $L(\mathscr X)$. Then $\mathrm{Fact}(J) = \ov{\iota(L(\mathscr X))}=\lambda^{-1}(S_{\mathscr X}\setminus \{0\})$ and hence is clopen.

Conversely, suppose $\mathrm{Fact}(J)$ is clopen.
Proposition~\ref{JclassesasFI} shows that $\mathrm{Fact}(J)$ is
factorial and irreducible. Then $L=\iota\inv (\mathrm{Fact}(J))$ is a
$\pv V$-recognizable language and
$\ov{\iota(L)}=\mathrm{Fact}(J)$~\cite{Almeida:book}; in particular,
$L\neq \emptyset$.  Moreover, it is factorial by Lemma~\ref{pullback}.
It remains to prove that $L$ is irreducible.
It will then follow that $L=L(\mathscr X)$ and $J=J(\mathscr X)$ for
an appropriate irreducible sofic shift $\mathscr X$.
Suppose $u,v\in L$ and $\iota(u)w\iota(v)\in \mathrm{Fact}(J)$ with $w\in \wh{F}_{\pv V}(X)$.  Then $w=\lim \iota(w_n)$ for some sequence $\{w_n\}\subseteq X^+$.  Then $\iota(uw_nv)\to \iota(u)w\iota(v)\in \mathrm{Fact}(J)$ and hence, as $\mathrm{Fact}(J)$ is open, for $n$ large enough $uw_nv\in \iota\inv(\mathrm{Fact}(J))=L$.  This completes the proof of irreducibility.
\end{proof}

An important lemma that we shall exploit frequently is the following.
\begin{Lemma}\label{mapsonto}
Let $\mathscr X\subseteq X^{\mathbb Z}$ be an irreducible sofic shift whose syntactic semigroup is contained in a variety of finite semigroups $\pv V$ and suppose we have a commutative diagram of continuous surjective morphisms
\[\xymatrix{\wh F_{\pv V}(X)\ar[r]^{\p}\ar[rd]_{\lambda} & S\ar[d]^{\psi}\\ & S_{\mathscr X}}\] where $\lambda\colon \wh F_{\pv V}(X)\to S_{\mathscr X}$ is the continuous extension of the syntactic morphism of $L(\mathscr X)$.  Then:
\begin{enumerate}
\item $\p(J(\mathscr X))$ is a regular $\J$-class $J$ of $S$;
\item $J$ is the unique minimal $\J$-class of $S$ with $\psi(J)$ contained in the distinguished $\J$-class of $S_{\mathscr X}$;
\item $J(\mathscr X)$ is the unique minimal $\J$-class of $\wh F_{\pv V}(X)$ mapping into $J$;
\item The image under $\p$ of each maximal subgroup of $J(\mathscr X)$ is a maximal subgroup of $J$.
\end{enumerate}
\end{Lemma}
\begin{proof}
Let $J_0$ be the distinguished $\J$-class of $S_{\mathscr X}$ and suppose that $J$ is the unique minimal $\J$-class of $T$ with $\psi(J)\subseteq J_0$ guaranteed by Lemma~\ref{liftclasses}.  Then $J$ is regular.  To complete the proof, it suffices by Lemma~\ref{liftclasses} to verify that $J(\mathscr X)$ is minimal among $\J$-classes of $\wh F_{\pv V}(X)$ mapping under $\p$ into $J$.  Suppose that $u\leq_{\J} J(\mathscr X)$ with $\p(u)\in J$.  Then $\psi(\p(u))\in J_0$ and hence $u\in \ov{L(\mathscr X)}$.  It follows that $u\in J(\mathscr X)$ by definition of $J(\mathscr X)$.
\end{proof}

The following lemma is an immediate consequence of
Lemma~\ref{reducedtonicecase}. The hypothesis
$\pv V=\pv V\ast \pv D$ is there to guarantee
that $J(\mathscr Y)$ is well defined, since
for two conjugate shifts
$\mathscr X$ and $\mathscr Y$,
the syntactic semigroup
of $L(\mathscr X)$ belongs to
$\pv V$ if and only if the syntactic semigroup
of $L(\mathscr Y)$ does~\cite{Alfredo1}.

\begin{Lemma}\label{realreduce}
Let $\mathscr X\subseteq X^{\mathbb Z}$ be an irreducible sofic shift
whose syntactic semigroup is contained in a variety of finite
semigroups $\pv V$ with
$\pv V=\pv V\ast \pv D$ and $\pv V$ containing all finite semilattices.  There there is a conjugate irreducible sofic shift $L(\mathscr Y)$ over an alphabet $Y$, an idempotent $e\in J(\mathscr Y)$ and a word $z\in Y^+$ so that $e=z^{\omega}e$ and $\mathsf{alph}(z)\subsetneq \mathsf{alph}(\mathscr Y)$.
\end{Lemma}
\begin{proof}
  Let $\mathscr Y$ and $z$ be as in Lemma~\ref{reducedtonicecase}.
  Then $z^{\omega}\in \ov{z^+}\subseteq \ov{L(\mathscr Y)}$ and so by
  minimality of $J(\mathscr Y)$, we can find $x,y\in \wh{F}_{\pv
    V}(Y)^1$ so that $xz^{\omega}y\in J(\mathscr Y)$. Since $J(\mathscr Y)$ is regular, there are idempotents $f,f'\in J(\mathscr Y)$ so that $fxz^{\omega}yf'= xz^{\omega}y$.  Consequently,
  $z^{\omega}yf'\in J(\mathscr Y)$.  By regularity of $J(\mathscr Y)$, we can then find an idempotent $e\in J(\mathscr Y)$ with $e\R z^{\omega}yf'$.  Then $z^{\omega}e=e$ as required.
\end{proof}

The set of idempotents in a profinite semigroup is closed and hence a profinite space.  It turns out that the idempotents in $\J$-classes corresponding to irreducible sofic shifts are dense in relatively free profinite semigroups.

\begin{Prop}\label{soficidempotentsaredense}
Let $\pv V$ be a variety of finite semigroups and $X$ a finite set.  Let $A$ be the set of idempotents of $\wh{F}_{\pv V}(X)$ that belong to a $\J$-class of the form $J(\mathscr X)$ for some irreducible sofic shift $\mathscr X\subseteq X^{\ZZ}$ with $L(\mathscr X)$ a $\pv V$-recognizable set.  Then $A$ is dense in $E(\wh{F}_{\pv V}(X))$,
\end{Prop}
\begin{proof}
Let $e\in E(\wh{F}_{\pv V}(X))$.  Then a basic neighborhood of $e$ is
of the form $\pi\inv(\pi(e))$ where $\pi\colon \wh{F}_{\pv V}(X)\to V$
is a continuous homomorphism to an element $V$ of $\pv V$.  Let $J$ be the $\J$-class of $\pi(e)$ and choose a minimal $\J$-class $J'$ of $\wh{F}_{\pv V}(X)$ with $\pi(J')\subseteq J$ as per Lemma~\ref{liftclasses}. In particular, $J'$ is regular and $\mathrm{Fact}(J')=\pi\inv(\mathrm{Fact}(J))$, and hence is clopen.  Proposition~\ref{classifyJshift} then implies that $J'=J(\mathscr X)$ for an irreducible sofic shift $\mathscr X$ with $L(\mathscr X)$ a $\pv V$-recognizable set.  By Lemma~\ref{liftclasses}, there is an idempotent $f\in J'=J(\mathscr X)$ with $\pi(f)=\pi(e)$.  Thus $f\in A\cap \pi\inv(\pi(e))$, establishing that $A$ is dense.
\end{proof}

\section{The Sch\"utzenberger representation and wreath products}\label{minimal}
In this section we collect a number of standard facts concerning
finite semigroups, which can be
found, for instance, in~\cite{CP,Arbib,qtheor}.

\subsubsection*{The Sch\"utzenberger representation}
Let $J$ be a regular $\J$-class of a finite semigroup $S$.  Fix an
$\R$-class $R$ of $J$.  Then $S$ acts on the right of $R$
by partial functions by simply restricting the  action of $S$ on
the right of itself.  More precisely, for $s\in S$ and $x\in R$,
define
\begin{equation*}
x\cdot s = \begin{cases} xs & xs\in R\\ \text{undefined}
  &\text{else.}\end{cases}
\end{equation*}

The resulting faithful partial transformation semigroup does not depend on $R$ up to isomorphism~\cite[Chapter 4, Section 6]{qtheor} and we denote it by $(R,\mathsf{RM}_J(S))$.  We use $\rho_J\colon S\to \RM_J(S)$ for
the associated quotient map.  The map $\rho_J$ is called
the (right) \emph{Sch\"utzenberger representation} of $S$ on $J$.  If $G$ is a maximal subgroup contained in $R$, then the restriction of the action of $G$ to $G\subseteq R$ is the regular representation and hence faithful.  Since $\RM_J(S)$ depends only on $J$ and not $R$, it follows that $\rho_J$ is injective on each maximal subgroup of $J$.  The results
of~\cite[Chapter 4, Section 6]{qtheor} imply that $\rho_J(J)$ is a regular $\J$-class of $\RM_J(S)$ and the Sch\"utzenberger
representation of $\rho_J(S)$ on it is faithful.

Retaining the above notation, denote
by $L(J)$ the set of $\eL$-classes of $S$ in $J$.
There is an action of $S$ by partial
transformations on $L(J)$ given by
\begin{equation}\label{RLMdef}
L_xs = \begin{cases} L_{xs} & xs\in J\\ \text{undefined}
  &\text{else.}\end{cases}
\end{equation}
The resulting faithful right partial transformation semigroup is
denoted by $(L(J),\RLM_J(S))$ and the quotient map by $\mu_J\colon S\to
\RLM_J(S)$.  See~\cite[Chapter 4, Section 6]{qtheor} for details.

\subsubsection*{Wreath products}
Let us briefly recall the
wreath product of partial transformation
semigroups~\cite{Eilenberg,qtheor}. In this paper, by a \emph{partial transformation semigroup}, we mean a pair $(B,S)$ where $S$ is a semigroup acting faithfully by partial transformations on the \emph{right} of $B$. If the maps in $S$ are total, then we use the terminology \emph{transformation semigroup}.  In the case that $S$ is a monoid (group) and the identity acts as the identity, then we say it is a \emph{partial transformation monoid (group)}.  If $B$ is a set, then $\ov B$ will denote the semigroup of all constant maps on $B$.

It will be convenient to use in this paper the formulation of wreath products in terms of row monomial matrices~\cite[Chapter 5]{qtheor} or~\cite{Arbib}.  If $S$ is a semigroup, then $S^0$ is the semigroup obtained by functorially adjoining a multiplicative zero $0$ (so a zero is added to $S$ even if it already had one).  Let $S$ be a non-empty semigroup and $B$ a set.  Then $RM_B(S)$ consists of all $B\times B$-matrices row monomial matrices over $S^0$  equipped with usual matrix multiplication. Recall that a matrix is \emph{row monomial} if each row contains at most one non-zero entry. The construction $RM_B(-)$ is functorial.   Suppose that $(B,T)$ is a partial transformation semigroup.  The full partial transformation monoid is easily seen to be isomorphic to $RM_B(\{1\})$~\cite[Chapter 5]{qtheor}.  Thus we may view $T$ as a subsemigroup of $RM_B(\{1\})$.  Let $\pi\colon RM_B(S)\to RM_B(\{1\})$ be the projection.  Then we define the \emph{wreath product} $S\wr (B,T) = \pi
 \inv (T)$. The projection $RM_B(S)\to RM_B(\{1\})$ restricts to a surjective morphism $S\wr (B,T)\to T$. Notice that $(-)\wr (B,T)$ is functorial and preserves surjective morphisms.  If $S$ is a group and $(B,T)$ is a transformation group, then $S\wr (B,T)$ is a group.

Let $J$ be a regular $\J$-class of a finite semigroup $S$ with maximal subgroup $G$.  Denote by $J^0$ the semigroup obtained by adding a multiplicative zero $0$ to $J$ and putting \[x\cdot y=\begin{cases} xy & xy\in J\\ 0 &\text{else.}\end{cases}\] for $x,y\in J$.  Then $J^0$ is $0$-simple and hence isomorphic to a Rees matrix semigroup $\MM^0(G,A,B,C)$ where
$C\colon B\times A\to G^0$ is the sandwich matrix~\cite{CP,Arbib,qtheor}.
Fix $a_0\in A$ and $b_0\in B$.  Then without loss of generality we
may assume that each non-zero entry of row $b_0$ and of column $a_0$ of $C$ is the
identity of $G$~\cite{Arbib,qtheor}.  We identify $G$ with the maximal
subgroup $a_0\times G\times b_0$.

Recall that $B$ can be identified with the set of $\eL$-classes of $J$~\cite{qtheor}.
Notice that each element of $J$
acts on $B$ as a rank $1$ partial map (cf.~\eqref{RLMdef}), where the \emph{rank} of a partial transformation is the size of its image.  Moreover, $J$ is transitive on $B$.

There is a well-known embedding $\RM_J(S)\hookrightarrow G\wr
(B,\mathsf{RLM}_J(S))$ such that an element
$s=(a,g,b)\in J$ is sent to the matrix $M(s)$ all of whose non-zero entries are in column $b$ and with $M(s)_{b'b}=
C_{b'a}g$~\cite[Proposition 4.6.42]{qtheor}. In particular, if
$s=(a_0,g,b_0)$ is an element of our maximal subgroup, then every non-zero entry of column $b_0$ of $M(s)$ is $g$ and in particular $M(s)_{b_0b_0}=g$, establishing yet again that the Sch\"utzenberger representation
is faithful on the maximal subgroup $G$.

The following well-known lemma elucidates the structure of wreath products.  See, for instance,~\cite[Theorems~9.3.10 and~9.3.15]{BerstelPerrinReutenauer}.

\begin{Lemma}\label{structureofwreath}
Let $S = G\wr (B,T)$ where $G$ is a non-trivial finite group and $T$ is a finite transitive partial transformation semigroup consisting of maps of rank at most $1$ and denote by $\pi\colon G\wr (B,T)\to T$ the wreath product projection.    Then $S$ is simple if $T$ consists of total maps and otherwise $S$ is $0$-simple. The maximal subgroup of the non-zero $\J$-class of $S$ is isomorphic to $G$.  More
precisely, if $e\neq 0$ is an idempotent such that the image of $\pi(e)$ is $\{b\}$, then
$\psi\colon G_e\to G$ given by $\psi(s)= s_{bb}$ is an isomorphism.
\end{Lemma}
\begin{proof}
First we claim that $T$ is simple if it consists of total maps and otherwise is $0$-simple.  If $T$ consists of total maps, it is a right zero semigroup and hence trivially simple.  Otherwise, suppose $0\neq t\in T$ is not total and that $b_0$ is not in the domain of $t$.  Let $\{b\}$ be the image of $t$.  By transitivity, there is an element $t'\in T$ with $bt'=b_0$.  Then $tt't=0$ so $0\in T$.  To see that $T$ is $0$-simple, suppose that $t,t'\in T$ are non-zero.  Let $b_t,b_{t'}$ be the unique elements in the image of $t, t'$, respectively.  Let $b$ be an element of the domain of $t'$.    By transitivity, we can find $u,v\in T$ with $b_tu=b$ and $b_{t'}v=b_t$.  Then $tut'v=t$ and so $t\in Tt'T$.  A symmetric argument shows that $t'\in TtT$.  Thus $T$ is $0$-simple.

Next we verify that $\pi(x)=\pi(y)$ implies $x\eL y$.  This will imply the simplicity or $0$-simplicity of $S$ (depending on which case we are in).  Let $f\in T$ be an idempotent that is $\R$-equivalent to $\pi(x)=t=\pi(y)$ ($T$ is regular).  Then $f$ has the same domain as $t$.  Let $z\in S$ be given by putting $z_{i,if} = y_{i,it}x_{if,it}\inv$ if $if$ (equivalently, $it$) is defined and $0$ otherwise.  Then one immediately verifies that $zx=y$.  A symmetric argument establishes that $x\eL y$.

Now suppose that $e$ is an idempotent
of $S$ such that $\pi(e)$ has image $\{b\}$. We must show $\psi$ defined as above is an isomorphism.  First note that every element $s$ of $S$ that is $\eL$-equivalent to $e$ satisfies $B\pi(s) =\{b\}$.  Thus each element of $G_e$ has all its non-zero entries in column $b$.  It now follows that if $s,s'\in G_e$, then $(ss')_{bb}=s_{bb}s'_{bb}$, that is, $\psi$ is a homomorphism.  In particular, we have $1=\psi(e)=e_{bb}$.

To see $\psi$ is injective, note that $s\in G_e$ implies
$s = es$.  Since $s,e$ have all their non-zero entries in column $b$, it follows that $s_{b'b} = e_{b'b}s_{bb}$ and so $s$ is determined by $\psi(s)=s_{bb}$.  Thus $\psi$ is injective. Finally to verify
$\psi$ is onto, assume $g\in G$.  Let $s$ be the element of $S$ obtained from $e$ by changing $e_{bb}$ to $g$ and leaving all other entries the same.   Then $\pi(ese) = \pi(e^3)=\pi(e)$ and so $ese\eL e$ by the above.  Hence $ese\HH e$.  Since $e_{bb}=1$, clearly $(ese)_{bb} = e_{bb}s_{bb}e_{bb} =g$.  Thus $\pi$ is onto.
\end{proof}

If $(A,S)$ and $(B,T)$ are partial transformation semigroups, then one has that $(A\times B,S\wr (B,T))$ is a partial transformation semigroup, which we denote $(A,S)\wr (B,T)$.  Here if $M\in S\wr (B,T)$, then $(a,b)M$ is defined if and only if $M_{bb'}\neq 0$ for some $b'\in B$ and $aM_{bb'}$ is defined.  The result is then $(aM_{bb'},b')$.  The wreath product of partial transformation semigroups is known to be associative~\cite{Eilenberg}.  We can view an iterated wreath product $S\wr (A,T)\wr
(B,U)$ as $|B|\times |B|$ block row monomial matrices
where the blocks are $|A|\times |A|$ row monomial matrices over $S$.
The term \emph{block entry} shall refer to a non-zero matrix from $S\wr (A,T)$
while the term \emph{entry} shall always mean an element of the
semigroup $S^0$. In general matrices, and in particular block entries,
shall be denoted by capital letters for the remainder of the paper.

\section{Statement of the main result and a reduction}
In this section, we state our main result and reduce its proof to a technical construction that will be presented in the next section.  Recall that a subset $Y$ of a profinite
group $G$ \emph{converges to the identity} if each neighborhood of the
identity contains all but finitely many elements of $Y$.  A pro-\pv
H group $F$ is \emph{free pro-\pv H on a subset $Y$ converging to the
identity} if given any map $\tau\colon Y\to H$ with $H$ pro-\pv H and
$\tau(Y)$ converging to the identity, there is a unique continuous extension of
$\tau$ to $F$. The cardinality of $Y$ is called the \emph{rank} of
$F$. Any free pro-\pv H group on a profinite space is free on a subset
converging to the identity~\cite{RZbook}.

The following theorem was proved in~\cite[Theorem 7.5]{AlmeidaVolkov2} for the case where $\ov{\pv H}$ consists of all finite semigroups. We provide here the proof of the general case.  Recall that $C\subseteq X^+$ is called a \emph{code} if $C^+$ is a free semigroup on $C$.  A code is said to have \emph{synchronization delay} at most $d$, if for all $(c,c')\in C^d\times C^d$ and all $x,y\in X^*$, one has $xcc'y\in C^+$ if and only if $xc,c'y\in C^+$.

\begin{Lemma}\label{synchdelay}
Let $u\in X^+$ be a primitive word.  Then, for any $m>0$ and any $x,y\in X^*$, one has $xu^my\in u^+$ if and only if $x,y\in u^*$.
\end{Lemma}
\begin{proof}
Suppose $xu^my\in u^+$.  Then we may write $x=u^kx'$ and $y=y'u^\ell$ so that $|x'|,|y'|<u$.  If we can show that $x'=1=y'$, we are done.  Suppose at least one of $x'$ and $y'$ are non-trivial.  Then since $x'u^my'\in u^+$, by length considerations we must have $x'u^my'=u^{m+1}$ and $|x'|+|y'|=|u|$.  But then $x'$ is a prefix of $u$ and $y'$ is a suffix of $u$ and so $x'y'=u$ by length considerations.  But then $x'u^my'=x'y'u^{m-1}x'y'$ and so $u^m=y'u^{m-1}x'$.  Thus $y'$ is a prefix and $x'$ is a suffix of $u$.  Length considerations then yield $y'x'=u=x'y'$.  But then $x',y'$ are powers of a word $w$ and hence $u$ is a proper power, contradicting primitivity.  Thus $x'=1=y'$, as was required. 
\end{proof}

\begin{Thm}\label{Theorem0}
Suppose that $X$ is a finite set and let $\mathscr X\subseteq
X^{\mathbb Z}$ be a periodic shift.  Let
$\pv H$ be a non-trivial variety of finite groups.
Then the maximal subgroup $G(\mathscr X)$ of \mbox{$J(\mathscr
  X)\subseteq \wh F_{\ov{\pv H}}(X)$} is a free pro-$\pv H$ group
of rank $1$.
\end{Thm}

\begin{proof}
Let $u$ be a primitive word such that $\mathscr X$ is the orbit
of $\ldots uuu.uuu\ldots$. 
We claim that the maximal subgroup $K$ of
\mbox{$J(\mathscr X)$} containing $u^\omega$ is generated by
$u^{\omega+1}$. This fact is a special case of
\cite[Theorem 7.5]{AlmeidaVolkov2}, but it can proved in an easier way. 
An elementary result of Restivo~\cite{purecode} shows that if $C\subseteq X^+$ is a code such that $C^+$ is pure, i.e., closed under extraction of roots, then the syntactic semigroup of $C^+$ is aperiodic (see also~\cite[Chapter~7, Exercise~8]{Lallement}). Clearly $u^+$ is closed under extraction of roots (by primitivity of $u$) and so its syntactic monoid $S$ is aperiodic and hence belongs to $\ov {\pv H}$.  If $\eta\colon \wh F_{\ov{\pv H}}(X)\to S$ is the canonical extension of the syntactic morphism, then $\eta(K)=\eta(u^{\omega})=\eta(u^+)$, where the first equality holds by aperiodicity, whereas the second is immediate from Lemma~\ref{synchdelay}.  We conclude $K\subseteq \eta\inv\eta(u^+) = \ov{u^+}$.  Since $\ov{\langle u^{\omega +1}\rangle}$ is the unique maximal subgroup of $\ov{u^+}$, this establishes the claim.

Since $K$ is procyclic, it is free pro-$\pv H$ if and only if every
cyclic group from $\pv H$ is an image of it.  Let $n$ be an integer such that $\pv H$ contains a cyclic group of order $n$.  Since $(u,u)$ is the unique pair in $\{u\}\times \{u\}$, Lemma~\ref{synchdelay} with $m=2$ immediately yields that the code $\{u\}$ has synchronizing delay at most $1$.  Therefore, if $T$ is the syntactic semigroup of $\{u^n\}^+$, then each maximal subgroup of $T$ is in the variety of finite groups generated by $\mathbb Z_n$, by~\cite[Chapter~7, Corollary~2.14]{Lallement} and hence $T\in \ov{\pv H}$ (this can also be deduced from $S$ being aperiodic and the result of Weil on subgroups of the syntactic semigroup of a composed code~\cite{pascalcode}).   Clearly, a necessary condition for $[u^{n+r}]_T=[u^{n+k}]_T$ is that $r\equiv k\bmod n$.  Consequently, $\langle [u]_T^{\omega+1}\rangle$ must in fact be a cyclic group of order $n$. 
Putting together what we have just shown, we see that $K=\ov{\langle u^{\omega+1}\rangle}$ maps onto the cyclic group of order $n$ whenever it belongs to $\pv H$. We conclude that $K$ is a free pro-$\pv H$
group of rank $1$. 
\end{proof}

  A fact that we shall use in the following theorem
  is that if $\pv V$ is a variety of
finite semigroups containing all finite semilattices such that
$\pv V=\pv V\ast \pv D$  and $\mathscr X,\mathscr
Y$ are conjugate sofic shifts, then the syntactic semigroup of
$L(\mathscr X)$ belongs to $\pv V$ if and only if the syntactic semigroup
of $L(\mathscr Y)$ does~\cite{Alfredo1}.

\begin{Thm}\label{Theorem1}
Suppose that $X$ is a finite set and let $\mathscr X\subseteq
X^{\mathbb Z}$ be an irreducible sofic shift.
Suppose that $\pv H$ is a variety of finite groups closed under
extension such that the syntactic semigroup
of $L(\mathscr X)$ belongs to $\ov{\pv H}$
and $\mathbb Z/p\mathbb Z\in \pv H$ for infinitely many primes $p$.
Then the maximal subgroup $G(\mathscr X)$ of \mbox{$J(\mathscr X)\subseteq \wh F_{\ov{\pv H}}(X)$} is a free pro-$\pv H$ group.  If $\mathscr X$ is minimal, then $G(\mathscr X)$ is procyclic; otherwise it is free pro-$\pv H$ of countable rank.
\end{Thm}
\begin{proof}
  By Theorem~\ref{Theorem0}, we may suppose that $\mathscr X$ is not
  minimal.  Because $\ov{\pv H} = \ov{\pv H}\ast \pv D$, the isomorphism class of the maximal subgroup of
  $J(\mathscr X)$ depends on $\mathscr X$ only up to conjugacy
  by~\cite{Alfredo2}.

  Lemma~\ref{realreduce}
  and the observation made in the paragraph before the theorem, allow
  us to assume that there exist an idempotent $e\in J(\mathscr X)$ and
  a word $z\in X^+$ so that $e=z^{\omega}e$ and
  $\mathsf{alph}(z)\subsetneq \mathsf{alph}(\mathscr X)$.  By possibly
  shrinking or enlarging the alphabet, we may assume without loss of
  generality that
  $X=\{x_1,\ldots,x_{n+1}\}$
  where $\mathsf{alph}(\mathscr X)=\{x_1,\ldots,x_n\}$ and that we have an idempotent $e\in J(\mathscr X)$ and a word $z\in \{x_1,\ldots, x_{n-1}\}^+$ so that $z^{\omega}e=e$ and $x_1\in \mathsf{alph}(z)$. Doing this lets us avoid treating the full shift as a special case.  Instead, we may assume that the syntactic semigroup $S_{\mathscr X}$ of $L(\mathscr X)$ has a zero element $0$.  Let $G_e$ be the maximal
subgroup at $e$. Our goal is to show that $G_e$ is free pro-\pv H on a
countable set of generators converging to the identity (that is,
free of countable rank).

It is well known $\ovFP H X$ is
metrizable~\cite{Almeida:book,qtheor}, and hence so is $G_e$.
Thus the identity $e$ of $G_e$ has a countable basis of
neighborhoods.  We shall use a well-known criterion, going back to
Iwasawa~\cite{Iwasawa}, to establish that $G_e$ is free pro-\pv H of
countable rank. An
\emph{embedding problem} for $G_e$ is a diagram
\begin{equation}\label{embeddingproblem}
\xymatrix{& G_e\ar@{->>}[d]^{\p}\\
H\ar@{->>}[r]^{\alpha}&K}
\end{equation}
with $H\in \pv H$ and $\p,\alpha$ epimorphisms ($\p$ continuous).

A \emph{solution} to the embedding problem \eqref{embeddingproblem}
is a continuous \textbf{epimorphism} $\til\p\colon G_e\to H$ making the diagram
\[\xymatrix{& G_e\ar@{-->>}[ld]_{\til\p}\ar@{->>}[d]^{\p}\\
H\ar@{->>}[r]^{\alpha}&K}\]
commute. (The terminology ``embedding problem'' comes from Galois theory.)
According to~\cite[Corollary 3.5.10]{RZbook} to prove
$G_e$ is free pro-\pv H of countable rank it suffices to show that
every embedding problem \eqref{embeddingproblem} for $G_e$ has a
solution. We proceed via a series of reductions on the types of
embedding problems we need to consider.  The initial reductions are
nearly identical to those in~\cite{projective,minimalideal}.

So let us suppose that we have an embedding problem for $G_e$ as per
\eqref{embeddingproblem}.  The reader is referred to~\cite[Chapter 3, Section 1]{qtheor} for basic properties of profinite semigroups and projective
limits; see also~\cite{RZbook} for the analogous results in the
context of profinite groups. Let $\lambda\colon \ovFP H X\to S_{\mathscr X}$ be the continuous extension of the syntactic morphism of $L(\mathscr X)$; note that $\lambda(x_{n+1})=0$.   Let $\{S_i\}_{i\in D}$ be the inverse
quotient system of all finite continuous images of $\ovFP H X$ such that $\lambda$ factors through the projection $\pi_i\colon\ovFP H X\to S_i$.  Then
$\ovFP H X = \ilim_{i\in D} S_i$.    Since $G_e$ is a closed subgroup of
$\ovFP H X$, it follows from basic properties of profinite spaces that
$G_e = \ilim_{i\in D}\pi_i(G_e)$ (see~\cite[Corollary 1.1.8]{RZbook}).  Since $\p$ is an onto continuous map to a finite group it follows that $\p$ factors through $\pi_i|_{G_i}$ for some
$i\in D$ (i.e.\ $\ker \pi_i|_{G_e}\subseteq \ker
\p$)~\cite[Lemma 1.1.16]{RZbook}.  Setting $S'=S_i$ and $\p'=\pi_i$, we
conclude there exists a continuous onto homomorphism $\p'\colon \ovFP
H X\twoheadrightarrow S'$ with $S'$ a finite semigroup in $\ov{\pv H}$
such that $\ker
\p'|_{G_e}\subseteq \ker \p$ and $\ker \p'\subseteq \ker \lambda$.

Set $K' = \p'(G_e)$ and let
$\rho\colon K'\twoheadrightarrow K$ be the canonical projection. Defining
$H'$ to be the pullback of $\alpha$ and $\rho$, that is, \[H' =
\{(h,k')\in H\times K'\mid \alpha(h)=\rho(k')\},\] yields a
commutative diagram
\begin{equation*}
\xymatrix{         &                 &     G_e\ar@{->>}@/_/[ld]_{\p'}\ar@{->>}@/^1pc/[ldd]^{\p}\\
H'\ar@{->>}[r]^{\alpha'}\ar@{->>}[d]_{\rho^*} & K'\ar@{->>}[d]^{\rho}&   \\
H\ar@{->>}[r]^{\alpha}  & K            &}
\end{equation*}
where $\rho^*$ is the projection to $H$ and $\alpha'$ is the projection to $K'$.  It is easily verified that
all the arrows in the diagram are epimorphisms.
So to solve our
original embedding problem, it suffices to solve the embedding
problem
\begin{equation}\label{betterdiagram}
\xymatrix{& G_e\ar@{->>}[d]^{\p'}\\
 H'\ar@{->>}[r]^{\alpha'}&K'}
\end{equation}
as the composition of a solution to \eqref{betterdiagram} with $\rho^*$ yields a solution to \eqref{embeddingproblem}.
In other words, reverting back to our original notation, we may
assume  in the embedding problem \eqref{embeddingproblem} that the
map $\p$ is the restriction of a continuous onto homomorphism
$\p\colon \ovFP H X\twoheadrightarrow S$ with $S\in \ov{\pv H}$ and $\ker \p\subseteq \ker \lambda$.

By Lemma~\ref{mapsonto}, $J=\p(J(\mathscr X))$ is a regular $\J$-class of $S$ and
the group $K=\p(G_e)$ is
a maximal subgroup of $J$.  By Section~\ref{minimal}, the right
Sch\"utz\-en\-ber\-ger representation $\rho_J\colon S\to \mathsf{RM}_J(S)$ of
$S$ on $J$ is faithful when restricted to $K$.  Moreover, if $\psi\colon S\to S_{\mathscr X}$ is the canonical projection, then $\ker \rho_J\subseteq \ker \psi$ by Lemma~\ref{mapsonto},~\cite[Proposition 4.6.37]{qtheor} and~\cite[Equation (4.8)]{qtheor} since $S_{\mathscr X}$ is an $\mathsf{AGGM}$-semigroup with distinguished $\J$-class $J_0$ and $J$ is minimal with $\psi(J)\subseteq J_0$.
Possibly replacing
$S$ by its image under the Sch\"utzenberger representation, we may then
assume that the right Sch\"utzenberger representation of $S$ on $J$
is faithful (recall that from the results
of~\cite[Chapter 4, Section 6]{qtheor} $\rho_J(J)$ is a regular $\J$-class of $\RM_J(S)$ and the Sch\"utzenberger
representation of $\rho_J(S)$ on it is faithful). Therefore, we may view $S$ as embedded in the wreath
product $K\wr (B,\mathsf{RLM}_J(S))$.  Moreover, $J$ is then the unique minimal non-zero $\J$-class of $S$~\cite[Proposition 4.6.29]{qtheor}.  Consequently, Lemma~\ref{mapsonto} implies the following lemma.

\begin{Lemma}\label{whenitdoesntfall}
  Let $u\in \ovFP H X$.
Then $u\in \ov{L(\mathscr X)}$ if and only if $\p(u)\neq 0$. In particular, $\p(x_{n+1})=0$.
\end{Lemma}

The existence of a solution to \eqref{embeddingproblem}
is then a consequence of the following technical lemma whose proof we defer to Section~\ref{s:proofofmaintech}.

\begin{Lemma}\label{maintechnical}
Let $\p\colon \ovFP H X\twoheadrightarrow S$ be a continuous surjective
morphism with $S$ finite and $\ker \p\subseteq \ker \lambda$ such that $\p(G_e) = K$ and the
Sch\"utzenberger representation of $S$ on the $\J$-class $J=\p(J(\mathscr X))$ is
faithful. In particular, $J$ is regular and is the unique minimal non-zero $\J$-class of $S$.  Suppose that $\alpha\colon H\twoheadrightarrow K$ is an epimorphism of finite groups.
 Then there is an $X$-generated finite semigroup $S'\in \ov{\pv H}$
 such that if $\eta\colon \ovFP H X\to S'$ is the continuous projection,
 then:
\begin{enumerate}
\item there is an isomorphism $\theta\colon G_{\eta(e)}\to H$ where
  $G_{\eta(e)}$ is the maximal subgroup of $S'$ at $\eta(e)$;
\item $\p$ factors through $\eta$ as $\rho\eta$ where
$\rho\colon S'\twoheadrightarrow S$ satisfies $\rho\theta\inv = \alpha$.
\end{enumerate}
\end{Lemma}

Assuming the lemma,  our
desired solution to the embedding problem \eqref{embeddingproblem} is
$\til{\p}=\theta \eta|_{G_e}\colon G_e\to H$.
Indeed, $\eta|_{G_e}\colon G_e\to G_{\eta(e)}$ is an epimorphism by
Lemma~\ref{mapsonto} (as $\ker \eta\subseteq \ker \p\subseteq \ker \lambda$) and hence $\til{\p}$ is an
epimorphism. Moreover, $\alpha\til{\p}=\rho\theta\inv
\theta\eta|_{G_e} = \p|_{G_e}$ and so $\til{\p}$ is indeed a
solution to the embedding problem \eqref{embeddingproblem}. This
completes the proof of Theorem~\ref{Theorem1}.
\end{proof}

Since the full shift is an irreducible sofic shift, an immediate corollary is the main result of~\cite{minimalideal} (although the proof of that result is simply a specialization of the current proof).

\begin{Cor}
Let $\pv H$ be a variety of finite groups closed under extension,
which contains $\mathbb Z/p\mathbb Z$ for infinitely may primes $p$.
Then the maximal subgroup of the minimal ideal of a finitely
generated (but not procyclic) free pro-${\ov {\pv H}}$ semigroup is a
free pro-\pv H group of countable rank.
\end{Cor}

It follows from our main result and Proposition~\ref{soficidempotentsaredense} that there is a dense set of idempotents in $\ovFP H X$ whose corresponding maximal subgroups are free pro-$\pv H$.

\section{The proof of Lemma~\ref{maintechnical}}\label{s:proofofmaintech}
 We retain the notation of the previous section.  In particular, recall that $X=\{x_1,\ldots,x_{n+1}\}$, there is a word $z\in \{x_1,\ldots,x_{n-1}\}^+$ so that $z^{\omega}e=e$, $x_1\in \mathsf{alph}(z)$,  $e$ is an idempotent of $J(\mathscr X)$ and $\mathsf{alph}(\mathscr X)=\{x_1,\ldots,x_n\}$.  Assume that $z=z_1\cdots z_q$ with $z_i\in \{x_1,\ldots,x_{n-1}\}$ for $i=1,\ldots, q$.

\begin{proof}[Proof of Lemma~\ref{maintechnical}]
Let $B$ be the set of $\eL$-classes of $J$.  Since we are assuming the Sch\"utzenberger representation of $S$ on $J$ is
faithful, we can view $S$ as a subsemigroup of $K\wr (B,\RLM_J(S))$, that is, as a semigroup of $b\times b$ row monomial
matrices over $K$ where $b=|B|$. Denote by $1$ the $\eL$-class of $\p(e)$.  We order the elements of $B$ with $1$ first when we write our matrices.   The discussion in
Section~\ref{minimal} shows that the embedding can be chosen so that the row monomial matrix associated to an element $k$ of the maximal
subgroup $K$ at $\p(e)$ has $k$ in every non-zero
entry of the first column and $0$ in the remaining columns.  Moreover, the $1,1$-entry of the row monomial matrix associated to $k$ is $k$.  For $x\in \ovFP H
X$, denote by $M_x$ the row monomial matrix associated to $\p(x)$.  We
shall distinguish formally between $M_x$ and $\p(x)$, although $M_x=M_y$ if and
only if $\p(x)=\p(y)$.

Let $N=\ker \alpha$ and choose a set-theoretic section $\sigma\colon K\to
H$.  Then $H=N\sigma(K)$.  For $x\in \ovFP H X$, denote by $M^{\sigma}_x$ the $b\times b$ row monomial
matrix over $H$ obtained from $M_x$ by applying $\sigma$ entry-wise.
Let $m$ be a positive integer such that
$(M^{\sigma}_{z_1}\cdots M^{\sigma}_{z_q})^m$ is idempotent in $H\wr
(B,\RLM_J(S))$. Choose a prime
$p>\max\{m,|N|^b,|z|_{x_1}\}$ so that $\mathbb Z/p\mathbb Z\in \pv H$; such a
prime exists by our assumption on \pv H.
Denote by $C_p$ the cyclic group of order $p$ generated by the
permutation $(1\ 2\cdots p)$. Our semigroup $S'$ will be a certain
subsemigroup of the iterated wreath product
\[Q=H\wr (B,\mathsf{RLM}_J(S))\wr \ov{([p],C_p)}\] where $[p]=\{1,\ldots,p\}$.  Observe that
$Q\in \ov{\pv H}$ since $\pv H$ closed under extension implies that
$\ov{\pv H}$ is closed under wreath product~\cite{Eilenberg,qtheor}. The reader is referred to~\cite{RhodesAllen} for more on taking a wreath product of a semigroup with a group with constant maps.

We begin our construction of $S'$ by defining
\begin{equation*}
\til x_1 = \begin{bmatrix} 0 & M_{x_1}^{\sigma}& 0              &\cdots & 0     \\
                            0 & 0               &M_{x_1}^{\sigma}&0      &\vdots \\
                            \vdots &    0               &  0             &\ddots & 0     \\
                            0 &    \vdots        &  0        & 0     & M_{x_1}^{\sigma}\\
                            M_{x_1}^{\sigma}& 0 &\cdots          &
                            0&0\end{bmatrix}.
\end{equation*}
In other words $\til x_1$ acts on the $[p]$-component by the cyclic
permutation $(1\ 2\ \cdots p)$ and each block entry of $\til x_1$
from $H\wr (B,\mathsf{RLM}_J(S))$ is $M_{x_1}^{\sigma}$. For $2\leq i\leq n-1$, we set
\begin{equation*}
\til x_i = \begin{bmatrix} M_{x_i}^{\sigma}& 0              &\cdots & 0     \\
                            0               &M_{x_i}^{\sigma}&0      &\vdots \\
                              \vdots               &  0             &\ddots & 0     \\
                             0               & \cdots         & 0     & M_{x_i}^{\sigma}\\
\end{bmatrix}.
\end{equation*}
So $\til x_i$ acts on the $[p]$-component as the identity map and each block entry of $\til x_i$ from $H\wr (B,\mathsf{RLM}_J(S))$ is $M_{x_i}^{\sigma}$, for $i=2,\ldots,n-1$.

To define $\til x_n$ will require some extra notation. Set
$\ell=|N|^b$; so $p>\ell$ by choice of $p$.
Let $1=N_1,N_2\ldots,N_\ell$ be the distinct elements of $N^b$. We
identify $N^b$ with the group of diagonal $b\times b$ matrices over
$N$.  In particular, $N^b$ is a subgroup of $H\wr
(B,\mathsf{RLM}_J(S)\cup \{1_B\})$.  In fact, there is a
natural onto homomorphism \[\ov{\alpha}\colon H\wr
(B,\mathsf{RLM}_J(S))\to K\wr (B,\mathsf{RLM}_J(S))\] induced by
$\alpha\colon H\to K$; the map $\ov{\alpha}$ simply applies $\alpha$
entry-wise. Moreover, it is straightforward to verify that
$\ov {\alpha}(U)=\ov {\alpha}(V)$ if and only if $U=N_jV$ some
$1\leq j\leq \ell$.  Indeed, if we denote by $u_i$ (respectively $v_i$)
the non-zero entry (if there is one) of row $i$ of $U$ (respectively $V$), then $\ov
{\alpha}(U)=\ov {\alpha}(V)$ implies $\alpha(u_i)=\alpha(v_i)$ for all
$i$ and so we can find $n_i\in N$ such that $u_i=n_iv_i$ for all $i$ (where if $i$ is a zero row of $u$ and $v$, then we may choose $n_i$ arbitrarily).
We may then take $N_j = \mathrm{diag}(n_1,n_2,\ldots,n_b)$. Dually,
$U=VN_k$, some $k$.

Next let us define a $p\times p$ block row
monomial matrix
\begin{equation*}
\til x_n = \begin{bmatrix} M_{x_n}^{\sigma}        & 0     &\cdots & 0     \\
                           N_2M_{x_n}^{\sigma}     & 0     &\cdots & 0 \\
                           \vdots                 & \vdots&\vdots & \vdots     \\
                           N_{\ell}M_{x_n}^{\sigma}&  0    &\cdots &   0   \\
                           M_{x_n}^{\sigma}        &  0    &\cdots &0 \\
                              \vdots               & \vdots&\vdots &\vdots\\
                           M_{x_n}^{\sigma}        &0      &\cdots &0   \end{bmatrix};
\end{equation*}
so $\til x_n$ has all its block entries in the first column.  The
$j^{th}$ block entry of the first column is $N_jM_{x_n}^{\sigma}$ if
$j\leq \ell$ and otherwise is $M_{x_n}^{\sigma}$. Finally, let $\til x_{n+1}=0$. Then $\til
x_1,\ldots,\til x_{n+1}\in Q$ and we have a map $X\to Q$ given by
$x_i\mapsto \til x_i$.  Extend this to a continuous morphism
$\eta\colon \ovFP H X\to Q$ and set $S'=\eta(\ovFP H X)$.  Our goal is to
show $S'$ is the desired semigroup.  We begin by verifying that $\p$
factors through $\eta$.

\begin{Prop}\label{formnumber1}
Let $u\in \ovFP H X$.  Then $\eta(u)=0$ if and only if $\p(u)=0$.
Moreover, if $\eta(u)\neq 0$, then $\eta(u)$ is a block $p\times
p$-matrix in which each block row contains a (non-zero) block
entry $U\in H\wr (B,\mathsf{RLM}_J(S))$,
and for every such block entry
one has $\ov{\alpha}(U)=M_u$.  As a consequence, $\eta(u)=\eta(u')$ implies
$\p(u)=\p(u')$ and so $\p$ factors through $\eta$ as $\rho\eta$
where $\rho\colon S'\to S$ takes $\eta(u)$ to $\ov{\alpha}(U)$ where $U$
is any block entry of $\eta(u)$.
\end{Prop}
\begin{proof}
The final statement follows from the previous ones since $\eta(u)=\eta(u')$
then implies $M_u=M_{u'}$ and so $\p(u)=\p(u')$.

We next prove the remaining part of the statement for words $u\in X^+$
by induction on
length, the case $|u|=1$ being trivial.
The result is also trivial for words containing $x_{n+1}$, so we only
deal with words not containing this element.
Suppose that $w=x_iu$ with $1\leq i\leq n$ and $u\in X^+$.
By induction $\p(u)=0$ if and only if $\eta(u)=0$
and so it only remains to deal with the case $\p(u)\neq 0$ and
$\eta(u)\neq 0$.  We recall that since the wreath product consists of
row monomial matrices, each block row of an element of $Q$
can have at most one block entry.

Let $1\leq j\leq p$.  By induction, $\eta(u)$ has a unique (non-zero) block entry $U_j$ in the $j^{th}$-block row.  The definition of $\til x_i$ implies that the $j^{th}$-block row
of $\eta(w)$ is obtained by multiplying each entry of a certain block
row $\xi(j)$ of $\eta(u)$ on the right by $N_{k_j}M_{x_i}^{\sigma}$
for some $N_{k_j}\in N^b$ (perhaps the identity).  So the only candidate to be a block entry of block row $j$ of $\eta(w)$ is $N_{k_j}M_{x_i}^{\sigma}U_{\xi(j)}$.
We claim that either $N_{k_j}M_{x_i}^{\sigma}U_{\xi(j)}$ is a (non-zero) block entry in the $j^{th}$-row of $\eta(w)$ for all $1\leq j\leq p$, or $\eta(w)=0=\p(w)$.

By induction, $\ov{\alpha}(U_{\xi(j)})=\p(u)$, thus we have
\begin{equation}\label{annoyingcasethatneedstobedone}
\ov{\alpha}(N_{k_j}M^{\sigma}_{x_i}U_{\xi(j)})=M_{x_i}\ov{\alpha}(U_{\xi(j)})=\p(x_i)\p(u)=\p(w)=M_w
\end{equation}
for $j=1,\ldots,p$.
Now the diagram \[\xymatrix{H\wr
  (B,\mathsf{RLM}_J(S))\ar[rr]^{\ov{\alpha}}\ar[rd] &&  K\wr
  (B,\mathsf{RLM}_J(S))\ar[ld]\\ & \mathsf{RLM}_J(S)&}\]
commutes, where the bottommost arrows are the wreath product
projections.
Since an element of a wreath product $L\wr (A,Z)$ is zero if and only
if its image under the wreath product projection is zero, it follows
from~\eqref{annoyingcasethatneedstobedone}
that $N_{k_j}M^{\sigma}_{x_i}U_{\xi(j)}=0$ for some $1\leq j\leq p$ if and only if $N_{k_j}M^{\sigma}_{x_i}U_{\xi(j)}=0$ for all $1\leq j\leq p$, if and only if
$\p(w)=0$.
Therefore, $\eta(w)=0$ if and only if $\p(w)=0$,
and if neither is $0$
then~\eqref{annoyingcasethatneedstobedone} implies that,
as a block $p\times p$-matrix, each
block row of $\eta(w)$ has exactly one (non-zero) block entry,
and each block entry is an  $\ov \alpha$-preimage of $M_w$.

If $u\in \ovFP H X$, then
since $X^+$ is dense
in $\ovFP H X$ and $\eta\inv\eta(u),\pinv\p(u)$ are open,
there exists a word $w\in X^+$ such that
$\eta(u)=\eta(w)$ and $\p(u)=\p(w)$.
The result now
follows from the case of words.
\end{proof}

It now follows that $\ker \eta\subseteq \ker \lambda$ and so
Lemma~\ref{mapsonto} yields  $J'=\eta(J(\mathscr X))$ is an entire
regular $\J$-class of $S'$.
Lemma~\ref{whenitdoesntfall} established $\p(u)=0$ if and only if
$u\notin \ov{L(\mathscr X)}$.  Thus by Proposition~\ref{formnumber1}
we conclude $\eta(u)=0$
if and only if $u\notin \ov{L(\mathscr X)}$ and hence $J'$ is the
unique minimal non-zero $\J$-class of $S'$.
Notice that $\rho(J')=\rho(\eta(J(\mathscr X))) = \p(J(\mathscr X)) =
J$.
In particular, $J'$ is minimal with $\rho(J')\subseteq J$ and so Lemma~\ref{liftclasses} applies.

 Our next goal is to show that if $u\in L(\mathscr X)$ is a word with $x_n\in \mathsf{alph}(u)$, then every preimage of $M_u$ under $\ov {\alpha}$ is a block entry of $\eta(u)$.  This will
be crucial in showing that the maximal subgroup of $\eta(J)$ is isomorphic to $H$.  To effect this we shall need the following
lemma.
Notice that if $U$ is any preimage of
$M_u$, then the complete set of preimages of $M_u$ is $\{N_1U,\ldots,
N_{\ell}U\} = \{UN_1,\ldots,UN_{\ell}\}$
(note that if $M_u$ has any zero rows, then these elements are not distinct).

\begin{Lemma}\label{preimages}
Let $u,w\in \ovFP H X$ and suppose $U$ is a fixed preimage
of $M_u$ under $\ov {\alpha}$. Then every preimage of $M_{uw}$ (respectively, $M_{wu}$) under $\ov{\alpha}$ is of the form $UW$ (respectively, $WU$) for some preimage $W$ of $M_w$ under $\ov{\alpha}$.
\end{Lemma}
\begin{proof}
Let $M$ be a preimage of $M_{uw}$ under $\ov{\alpha}$.
Since $UM_w^{\sigma}$ is a preimage of $M_{uw}$ under $\ov{\alpha}$, it follows that $M=UM_w^{\sigma}N_i$ for some $1\leq i\leq \ell$.  But then $W=M_w^{\sigma}N_i$ is an $\ov{\alpha}$-preimage of $M_w$ and $M=UW$.  The statement for $M_{wu}$ is proved dually.
\end{proof}

Observe that if $w\in X^+$ and $x_n\in \mathsf{alph}(w)$, then by definition of $\til x_1,\ldots,\til x_n$,
the block entries of $\eta(w)$ form a single column, in other words, the
$\ov{([p],C_p)}$-component of $\eta(w)$ is a constant map.
We can now prove the aforementioned fact concerning preimages.

\begin{Prop}\label{allthere}
Let $w\in L(\mathscr X)$ with $x_n\in \mathsf{alph}(w)$.  Then the set of
preimages of $M_w$ under $\ov {\alpha}$ is the set of block entries of
$\eta(w)$.
\end{Prop}
\begin{proof}
Let $R$ be the set of words $w\in L(\mathscr X)$ with $x_n\in
\mathsf{alph}(w)$.
We proceed by induction on $|w|$ for $w\in R$. If $|w|=1$, then the proposition follows from the definition of $\til x_n$.

Suppose it is true for words in $R$ of length $n$ and let $w\in R$
have length $n+1$.
Let $W$ be a $\ov{\alpha}$-preimage of $M_w$.
If the first
letter of $w$ is $x_n$, then $w=vx_i$ with $v\in R$, some $i$; else $w=x_iv$
where $v\in R$ and $1\leq i\leq n-1$.
In the latter case, by Lemma~\ref{preimages}
we have
$W=M_{x_i}^{\sigma}V$ for some $\ov{\alpha}$-preimage $V$ of $M_v$.
By induction hypothesis, $V$ is a block entry of
$\eta(v)$. Then, since $\eta(w)=\eta(x_i)\eta(v)$,
it follows from the definition of $\eta(x_i)$
that $M_{x_i}^{\sigma}V$
is a block entry of $\eta(w)$.

In the case $w=vx_i$ for some $v\in R$ and some $i$, the block entries of
$\eta(v)$ are in a single column, say column $j$.  Let $U$ be the block
entry in row $j$ of $\til x_i$; by construction it is an $\ov{\alpha}$-preimage of
$M_{x_i}$. By Lemma~\ref{preimages}
we have $W=VU$
for some $\ov{\alpha}$-preimage $V$ of $M_v$.
By induction hypothesis, $V$ is a block entry of
$\eta(v)$, and so it is in column $j$ of $\eta(v)$. Hence
$VU$ is a block entry of $\eta(w)$.
\end{proof}

A continuity argument allows us to extend the above result beyond words.

\begin{Cor}\label{minimalidealguyallthere}
If $w\in J(\mathscr X)$, then the block entries of $\eta(w)$ are in a single
column and the set of preimages under $\ov{\alpha}$ of $M_w$ is the set of
block entries of $\eta(w)$.
\end{Cor}
\begin{proof}
Consider the continuous homomorphism $c\colon \ovFP H X\to (P(X),\cup)$ defined by setting $c(x)= \{x\}$ for $x\in X$. Recall that we are assuming that $\{x_1,\ldots,x_n\}= \mathsf{alph}(\mathscr X)$.  Since $L(\mathscr X)$ is irreducible, we can find words $v_1,\ldots, v_{n-1}$ so that $x_1v_1x_2\cdots x_{n-1}v_{n-1}x_n\in L(\mathscr X)$.  It follows that $\mathsf{alph}(\mathscr X)\in c(L(\mathscr X)) = c(\ov {L(\mathscr X)})$ (the latter by continuity).  By minimality of $J(\mathscr X)$, we conclude that $c(J(\mathscr X)) = \{\mathsf{alph}(\mathscr X)\}$.  Thus if $\{w_r\}$ is a sequence of words in $X^+$ converging to
$w$, then there exists $R>0$ such that, for $r\geq R$, we have $\mathsf{alph}(w_r)=c(w_r) = \mathsf{alph}(\mathscr X)$.  The semigroup $S'$ is finite so there exists $s\geq R$ with
$\eta(w)=\eta(w_s)$ by continuity of $\eta$. Remembering that $\p =
\rho\eta$, this implies that $\p(w)=\p(w_s)$, or equivalently that
$M_w=M_{w_s}$.   As $\mathsf{alph}(w_s)=\mathsf{alph}(\mathscr X)=\{x_1,\ldots,x_n\}$, the corollary now follows from Proposition~\ref{allthere} and
the remark preceding that proposition applied to $w_s$.
\end{proof}

Recalling that $\mu_J\colon S\to \RLM_J(S)$ denotes the canonical projection, observe that $T=\mu_J(J\cup \{0\})$ is a transitive semigroup of  partial transformations of $B$ of rank at most $1$ containing the empty map. By Corollary~\ref{minimalidealguyallthere} if $w\in J(\mathscr X)$, then the
$\ov{([p],C_p)}$-component of $\eta(w)$ is a constant map, that is,
the block entries of $\eta(w)$ appear in a single column. Moreover,
Proposition~\ref{formnumber1} shows that each block entry of
$\eta(w)$ is a preimage of $M_w$ under $\ov{\alpha}$. Hence, $J'\subseteq H\wr (B,T)\wr
([p],\ov{[p]})$. Moreover, $(B,T)\wr
([p],\ov{[p]})$ is easily verified to be a transitive semigroup of partial transformations of rank at most $1$.  Indeed,  each entry of an element of $(B,T)\wr
([p],\ov{[p]})$ has all of its block entries in a single column and each non-zero element of $T$ is a rank $1$ partial transformation. The transitivity is immediate from the transitivity of $T$ and $\ov{[p]}$ since if $(b,i),(b',j)\in B\times [p]$  and $bt=b'$ with $t\in T$, then $(b,i)D=(b',j)$ where the block entries of $D$ are all in column $j$ and each block entry of $D$ is $t$.    Lemma~\ref{structureofwreath} now implies that $H\wr (B,T)\wr
([p],\ov{[p]})$ is $0$-simple.

It remains to construct an
isomorphism $\theta\colon G_{\eta(e)}\to H$ such that
$\rho\theta\inv=\alpha$.
Corollary~\ref{minimalidealguyallthere} yields that all the block entries of $\eta(e)$ belong to a single column.  By cyclically permuting the names of the elements of $[p]$, we may assume without loss of generality that
$\eta(e)$ is a block matrix with each block entry in the first
column.  Also, the discussion in Section~\ref{minimal} indicates
$M_e$ is a matrix whose only non-zero column is the first column and
whose non-zero entries are comprised of the identity of $K$; moreover, the $1,1$-entry of $M_e$ is the identity of $K$. Since the block entries of $\eta(e)$ are preimages of $M_e$ under
$\ov{\alpha}$ (Proposition~\ref{formnumber1}), we deduce that all the
non-zero entries of $\eta(e)$ are in the
first column and belong to $N$. Lemma~\ref{structureofwreath} says that
the map $\Theta\colon H\wr (B,T)\wr ([p],\ov{[p]})\to H$ selecting the
$1,1$-entry of a matrix is an isomorphism from the maximal subgroup at $\eta(e)$
of $H\wr (B,T)\wr ([p],\ov{[p]})$ to $H$.  In particular,
$\eta(e)_{11}$ is the identity of $H$. We shall show that the restriction $\theta$ of $\Theta$ to $G_{\eta(e)}$ is onto and $\rho\theta\inv=\alpha$.  This will require a little preparation.

\begin{Prop}\label{describerho}
If $u\in G_e$, the $\p(u)=\alpha(\eta(u)_{11})$.
\end{Prop}
\begin{proof}
Corollary~\ref{minimalidealguyallthere} implies that all the block entries of $\eta(u)$ are in a single column.  In fact, they are all in the first column since we just saw that this is the case for $\eta(e)$ and $\eta(u)=\eta(u)\eta(e)$.  Proposition~\ref{formnumber1} implies that $M_u$ is the matrix obtained by choosing any block entry of $\eta(u)$ and applying $\ov{\alpha}$.  In particular, $M_u$ is the result of applying $\alpha$ entry-wise to the $1,1$-block entry of $\eta(u)$ and so $[M_u]_{11} = \alpha(\eta(u)_{11})$.

Now if $k\in K$, then according to first paragraph of the proof of Lemma~\ref{maintechnical} the row monomial matrix associated to $k$ has $1,1$-entry $k$.  In particular, since $\p(u)\in K$, it follows that $[M_u]_{11}=\p(u)$.  The last statement of the previous paragraph then yields $\p(u)=\alpha(\eta(u)_{11})$, as required.
\end{proof}

The proposition admits the following corollary.

\begin{Cor}\label{workontheta}
The equality $\alpha\theta=\rho|_{G_{\eta(e)}}$ holds.
\end{Cor}
\begin{proof}
Recalling that $\theta$ selects the $1,1$-entry of an element of $G_{\eta(e)}$, Proposition~\ref{describerho} shows that $\p = \alpha\theta\eta$ as maps from $G_e$ to $K$.  By definition of $\rho$ there is a factorization $\p = \rho\eta$ and hence, in fact, $\rho\eta=\alpha\theta\eta\colon G_e\to K$.  But $\eta(G_e)=G_{\eta(e)}$ by Lemma~\ref{mapsonto}, so we conclude that $\rho|_{G_{\eta(e)}} =\alpha\theta$ as was to be proved.
\end{proof}

Since $\theta$ is injective, being a restriction of the isomorphism $\Theta$,
Corollary~\ref{workontheta} immediately yields that if $\theta$ is onto, then $\alpha = \rho\theta\inv$.  Thus we are left with proving $\theta$ is onto.  Since $J'$ is the minimal $\J$-class with $\rho(J')\subseteq J$ (as was already observed) $\rho$ must take $G_{\eta(e)}$ onto $K$ by Lemma~\ref{liftclasses}.  It
follows from Corollary~\ref{workontheta} that $\alpha$ maps $\theta(G_{\eta(e)})$ onto $K$.  Recalling $\ker\alpha =N$, we conclude $H=N\theta(G_{\eta(e)})$ and so to complete the proof it suffices to establish that $N$ is contained in the image of $\theta$.

Recall that we have a word $z=z_1\cdots z_q$ such that $z_i\in \{x_1,\ldots, x_{n-1}\}^+$, for $i=1,\ldots,q$, the letter $x_1$ is a factor of $z$  and $z^{\omega}e=e$. Set $Z=M^{\sigma}_{z_1}\cdots M^{\sigma}_{z_q}\in H\wr (B,\RLM_J(S))$. Let us remind the reader that our prime $p$ was chosen so that $p>m$ where $Z^m = Z^{\omega}$ and $p>|z|_{x_1}$.  Thus, we can find a positive integer
$r$ so that $1\equiv rm|z|_{x_1}\bmod p$.
Direct computation shows that \[\eta(z)^{rm} = \begin{bmatrix} 0 & Z^{\omega}& 0              &\cdots & 0     \\
                            0 & 0               &Z^{\omega}&0      &\cdots \\
                            0 & 0               &  0             &\ddots & 0     \\
                            0 & 0               & \cdots         & 0     & Z^{\omega}\\
                            Z^{\omega}& 0 &\cdots          &
                            0&0\end{bmatrix}\]
since every block entry of $\eta(z)$ is $Z$ and $\eta(z)$ acts in the $[p]$-component by the permutation $(1\ 2\ \cdots p)^{|z|_{x_1}}$.
Set $C=\eta(z)^{rm}$.  Then $C^j$ has the block form of the
permutation matrix corresponding to $(1\ 2\ \cdots p)^j$ and each
block entry of $C^j$ is $Z^{\omega}$.  In particular, the effect of multiplying a matrix $D$ on the left by $C^j$ is to permute the rows of $D$ according to the permutation $(1\ 2\ \cdots p)^{-j}$ and to multiply each row of $D$ on the left by $Z^{\omega}$.

Notice that $\eta(e)=\eta(z)^{\omega}\eta(e)$ as $z^{\omega}e=e$.  But $\eta(z)^{\omega}$ is a $p\times p$-block diagonal matrix with $Z^{\omega}$ as each diagonal block.  Thus the set of elements of the form $Z^{\omega}U$ with $U$ a block entry of $\eta(e)$ is precisely the set of block entries of $\eta(e)$, which is precisely the set of preimages of $M_e$ under $\overline{\alpha}$ by Corollary~\ref{minimalidealguyallthere}. Each preimage of $M_e$ is therefore the $1,1$-block entry of a product $C^j\eta(e)$ for a correctly chosen $j$ as $(1\ 2\ \cdots p)$ acts transitively on $\{1,\ldots,p\}$ and all the block entries of $\eta(e)$ are in the first column.

Now $M_e$ is a matrix all of whose non-zero entries are the identity of $K$ and belong to the first column;  moreover, the $1,1$-entry of $M_e$ is the identity of $K$. It follows that the $\ov{\alpha}$-preimages of $M_e$ are precisely
those matrices with first column having entries from $N=\ker \alpha$ in those rows that are non-zero in $M_e$ and whose remaining columns consist of zeroes.  Consequently, any element of
$N$ can be the $1,1$-entry of an $\ov{\alpha}$-preimage of $M_e$ and
so every element $h\in N$ is $[C^j\eta(e)]_{11}$ for some $j$ (as an entry, not a block entry).
Since $\eta(e)_{11}$ is the identity of $H$, it follows
$\eta(e)C^j\eta(e)$ has
$1,1$-entry $h$, and in particular is a non-zero element of $S'$.  Thus $\eta(e)C^j\eta(e)$ is an element of $G_{\eta(e)}$ by minimality of $J'$ among non-zero $\J$-classes of $S'$.  By construction, $\theta(\eta(e)C^j\eta(e))=h$ and so $\theta(G_{\eta(e)})$ contains $N$
as required.   This completes the proof of Lemma~\ref{maintechnical},
thereby establishing Theorem~\ref{Theorem1}.
\end{proof}

\section{Computing idempotents in the $\J$-class of a sofic shift}
Let $X$ be a finite alphabet.  An element $w$ of $\wh {X^+}$
is said to be \emph{(polynomial time) computable}
if there is an algorithm which on input an $X$-tuple $(s_x)_{x\in X}$ of elements from a finite semigroup $S$, computes (in time polynomial in $|S|$) the value $\p(w)$ where $\p\colon \wh {X^+}\to S$ is the canonical extension of the map $x\mapsto s_x$.  The existence of a computable idempotent in the minimal ideal of $\wh{X^+}$ was proved independently by Reilly and Zhang~\cite{ReillyZhangidempotent} on the one hand, and by Almeida and Volkov~\cite{AlmeidaVolkov} on the other.
The Reilly-Zhang idempotent was shown to be polynomial time computable in~\cite{AlmeidaVolkov}.

Let $\mathscr X$ be an irreducible sofic shift
over an alphabet $X$. Our goal is to construct a computable idempotent in $J(\mathscr X)$.  In fact, we give an algorithm which is uniform in the sofic shift, given as input via a so-called irreducible presentation.  First we need a lemma.

\begin{Lemma}\label{bottomcriterion}
Let $\mathscr X\subseteq X^{\mathbb Z}$ be an irreducible sofic shift whose syntactic semigroup is in a variety $\pv V$ of finite semigroups.  Let $\iota\colon X^+\to \wh F_{\pv V}(X)$ be the canonical morphism.  Then $u\in \ov{\iota(L(\mathscr X))}$ belongs to $J(\mathscr X)$ if and only if each element of $\iota(L(\mathscr X))$ is a factor of $u$.
\end{Lemma}
\begin{proof}
Clearly each element of $\iota(L(\mathscr X))$ is a factor of each element of
$J(\mathscr X)$ by minimality of $J(\mathscr X)$.
Suppose conversely, that each element of $\iota(L(\mathscr X))$ is a factor
of
$u\in \ov{\iota(L(\mathscr X))}$.
Let $v\in J(\mathscr X)$ and write $v=\lim \iota(v_n)$ with the $v_n$ in
$X^+$.
Since
$\ov{\iota(L(\mathscr X))}$ is clopen, it follows that, for all $n$ sufficiently large, $\iota(v_n)\in \ov{\iota(L(\mathscr X))}\cap \iota(X^+) = \iota(L(\mathscr X))$  (since $L(\mathscr X)$ is $\pv V$-recognizable~\cite{Almeida:book}).  Hence we may assume that $v_n\in L(\mathscr X)$ for all $n$.  By hypothesis on $u$, we can find, for each $n$, elements $r_n,s_n\in \wh F_{\pv V}(X)$ so that $r_n\iota(v_n)s_n=u$.  Passing to a subsequence, we may assume that $r_n\to r, s_n\to s$ for some $r,s\in \wh F_{\pv V}(X)$.  Then $u=rvs$ and so $u\in J(\mathscr X)$ by minimality of $J(\mathscr X)$.
\end{proof}

The following can be found in~\cite[Chapter 3]{MarcusandLind}.  For an
irreducible sofic shift
$\mathscr X$ there is a
strongly connected graph $\Gamma$
with non-empty set $E$ of edges,
and a map $\pi\colon E\to X$ such that,
if $P$ is the set of paths in $\Gamma$ then,
denoting the unique extension of
$\pi$ to a homomorphism $E^+\to X^+$ also by $\pi$,
we have $\pi(P)=L(\mathscr X)$. 
We say that $(\Gamma,\pi)$ is
an \emph{irreducible presentation} of
$\mathscr X$.  In other words, an irreducible presentation is a strongly connected non-deterministic automaton, all of whose states are initial and final, recognizing $L(\mathscr X)$.

We shall reduce our problem to producing a computable idempotent in the kernel of a clopen subsemigroup of $\wh {X^+}$.

\begin{Lemma}\label{reducetoratsubsemigroup}
Let $(\Gamma,\pi)$ be an irreducible presentation of a sofic shift $\mathscr X\subseteq X^{\mathbb Z}$ and let $v$ be a vertex of $\Gamma$.  Let $T\subseteq X^+$ be the rational subsemigroup of all words in $X^+$ reading a loop at $v$.  Then each element of the minimal ideal of $\ov T\subseteq \wh{X^+}$ belongs to $J(\mathscr X)$.
\end{Lemma}
\begin{proof}
By construction, $T\subseteq L(\mathscr X)$ and hence $\ov T\subseteq \ov{L(\mathscr X)}$.
Suppose $t$ belongs to the minimal ideal of $\ov T$.  Then each element of $T$ is a factor of $t$.  But since $\Gamma$ is strongly connected, each word labeling a path in $\Gamma$ is a factor of an element of $T$.  Thus $t\in J(\mathscr X)$ by Lemma~\ref{bottomcriterion}.
\end{proof}

In light of Lemma~\ref{reducetoratsubsemigroup}, to achieve our goal, it suffices to construct a computable idempotent $\rho_T$ in the minimal ideal of $\ov{T}$ for each rational subsemigroup $T\subseteq X^+$.  Moreover, our algorithm will be uniform in $T$ (meaning, given an automaton for $T$ and an $X$-tuple of a finite semigroup $S$, we can compute the value of $\rho_T$ on this $X$-tuple).

\begin{Lemma}\label{rationalsubsetbound}
Let $L\subseteq X^+$ be a rational subset and $\p\colon X^+\to S$ a homomorphism.  Then each element in $\p(L)$ can be represented by a word in $L$ of length at most $m(|S|+1)-1$ where $m$ is the number of states of the minimal automaton for $L$.
\end{Lemma}
\begin{proof}
Let $\mathscr A=(Q,X,\delta,q_0,F)$ be the minimal automaton for $L$ and construct an automaton $\mathscr B=(Q\times S^1,X,\Delta,(q_0,1),F\times S)$ where $(q,s)x=(qx,s\p(x))$ describes the transition function.  Then $s\in \p(L)$ if and only if there is a word $w\in X^+$ reading in $\mathscr B$ from $(q_0,1)$ to a state of the form $(q,s)$ with $q\in F$.   Since $\mathscr B$ has $m(|S|+1)$ states, such a word can always be chosen to have length at most $m(|S|+1)-1$.
\end{proof}

We can now construct our idempotent using the `Zimin word' idea of Reilly and Zhang (see also~\cite{AlmeidaVolkov}).  We assume that the alphabet $X$ is totally ordered.  The \emph{shortlex} order is defined on $X^+$ by putting $u<v$ if $|u|<|v|$ or $|u|=|v|$ and $u$ lexicographically precedes $v$.

\begin{Thm}\label{computableidempotent}
Let $T\subseteq X^+$ be a rational subsemigroup.  Let $v_1,v_2,\ldots$ be the elements of $T$ in shortlex order and put $w_1=v_1$ and $w_{n+1}=(w_nv_{n+1}w_n)^{(n+1)!}$ for $n\geq 1$.  Then the sequence $\{w_n\}$ converges to a computable idempotent $\rho_T$ of the minimal ideal of $\ov {T}\subseteq \wh{X^+}$.
\end{Thm}
\begin{proof}
First of all, since $T$ has decidable membership and there is a Turing machine that enumerates $X^+$ in shortlex order, clearly there is a Turing machine that can compute the element $w_n$ of the sequence given $n$ as input.  Let $\p\colon X^+\to S$ be a morphism where $S$ is a finite semigroup of order $k$.  Let $m$ be the number of states of the minimal automaton of $T$.  Put $r=m(k+1)-1$ and $N=|X|+|X|^2+\cdots+|X|^r$.  We claim that $\p(w_N) = \p(w_n)$ for all $n\geq N$ and that $\p(w_N)$ is an idempotent of the minimal ideal of $\p(T)$.  It will then follow that $\{w_n\}$ converges to a computable idempotent in the minimal ideal of $\ov{T}$.

First observe that for $n\geq k$, the elements $\p(w_n)$ form a descending chain of idempotents.  Now, by choice of $N$, for $n\geq N$, every word in $T$ of length at most $r$ is a factor of $w_n$.  Lemma~\ref{rationalsubsetbound} then yields that every element of $\p(T)$ is a factor of $\p(w_n)$; consequently, $\p(w_n)$ is an element of the minimal ideal $I$ of $\p(T)$. But $I$ is a completely simple semigroup and so contains no strictly descending chains of idempotents.  Thus $\p(w_n)=\p(w_N)$ for all $n\geq N$.  This completes the proof.
\end{proof}

The proof shows the construction is uniform in $T$.  Applying Lemma~\ref{reducetoratsubsemigroup}, we obtain:

\begin{Cor}
If $\mathscr X\subseteq X^{\mathbb Z}$ is an irreducible sofic shift, then there is a computable idempotent in $J(\mathscr X)$.
\end{Cor}

We leave it as an open question whether the polynomial time algorithm in~\cite{AlmeidaVolkov} to compute the Reilly-Zhang idempotent (which is our idempotent for $T=X^+$) can be extended to arbitrary rational subsemigroups.

\section{Entropy}\label{sec:entropy}

  Let $\mathscr X$ be a shift of $X^{\mathbb Z}$.
  The \emph{complexity function of $\mathscr X$}
  is the map $q_{\mathscr X}$
  that assigns to each positive integer $n$ the number
  of elements of $L(\mathscr X)$ with length $n$.
  This map satisfies
  the property
  $q_{\mathscr X}(n+m)\leq q_{\mathscr X}(n)\cdot q_{\mathscr X}(m)$.
  As proved in~\cite[Lemma~4.1.7]{MarcusandLind},
  this property implies the convergence of
  the sequence $\{\frac{1}{n}\log_{2}q_{\mathscr X}(n)\}$
  to its infimum $h(\mathscr X)$, called the \emph{entropy of~$\mathscr X$}.
  Complexity functions and entropy are
  fundamental notions in symbolic dynamics.
  In~\cite{AlmeidaVolkov2} these notions were adapted
  to the elements of $\wh{F}_{\pv V}(X)$,
  where $\pv V$ is a variety of finite semigroups containing
  $\mathbf {LSl}$. More precisely,
  the \emph{complexity function of an element $w$ of $\wh{F}_{\pv V}(X)$}
  is the map $q_w$
  that assigns to each positive integer $n$ the number
  of finite factors of $w$ with length $n$; this map
  also satisfies
  $q_{w}(n+m)\leq q_{w}(n)\cdot q_{w}(m)$,
  and so, if $w\notin X^+$, the sequence
  $\{\frac{1}{n}\log_{2}q_{w}(n)\}$
  converges to its infimum $h(w)$, called the \emph{entropy of $w$}.
  The entropy of elements of $X^\ast$ is defined to be $0$.

  One should be more precise and say that
  in~\cite{AlmeidaVolkov2}
  the entropy of $w\in \wh{F}_{\pv V}(X)\setminus X^+$
  is defined as the limit of
  $\{\frac{1}{n}\log_{|X|}q_{w}(n)\}$,
  which is $h(w)\log_{|X|}2$ according to our definition of
  $h(w)$.
  The two definitions are essentially the same,
  but ours does not depend on the alphabet,
  and it is more consistent with the usual definition of entropy of
  a shift.

  Next we recall from~\cite{AlmeidaVolkov2} some properties
  about the entropy of elements of $\wh F_{\pv V}(X)$, starting with the
  following:
  \begin{equation}\label{eq:entropy-and-multiplication}
    h(uv)=\max\{h(u),h(v)\}.
  \end{equation}
  In particular,
  the set $S_k$ of elements with
  entropy  less than $k$ is a subsemigroup
  of $\wh F_{\pv V}(X)$.
    It is well know that
  the elements of
  relatively free profinite semigroups have an operational
  interpretation. Under this interpretation,
  these elements are called \emph{implicit operations}; see~\cite{Almeida:book} for details.
  We may compose implicit operations.
  Then~\eqref{eq:entropy-and-multiplication} has the following
  generalization: if
    $w$ is a $r$-ary implicit operation
    over $\pv V$, then
    \begin{equation}\label{eq:entropy-and-composition}
    h(w(v_1,\ldots,v_r))\leq \max\{h(w)\cdot \log_{|X|}2\cdot  \log_{|X|}r,
    h(v_1),\ldots, h(v_r)\}.
    \end{equation}

  If $S$ is a finitely generated profinite semigroup, then
  the monoid $\mathrm {End}(S)$
  of continuous endomorphisms of $S$
  is a profinite monoid,
  considered with
  the pointwise topology,
  which coincides with the compact-open topology~\cite{Almeida:2005bshort}.
  For each $\p\in \mathrm {End}(S)$, elements of the
  subsemigroup of $\mathrm {End}(S)$
  generated by $\p$
  are of the form $\p^{\nu}$,
  where the exponent $\nu$ is an element of the profinite completion
  $\wh{\mathbb N}$
  of $\mathbb N$ (details can be found in~\cite{AlmeidaVolkov2}).
  It was proved in~\cite{AlmeidaVolkov2} that
  \begin{equation}
    \label{eq:entropy-and-iteration}
  \max_{x\in X}h(\p^{\nu}(x))\leq \max_{x\in X}h(\p(x))
  \end{equation}
   for every $w\in \wh F_{\pv V}(X)$ and $\nu\in \widetilde{\NN}\setminus\{0\}$.
  A subset $W$ of $\wh F_{\pv V}(X)$
  is \emph{closed under iterations}
  if $\p(x)\in W$ implies $\p^{\nu}(x)\in W$,
  for all $\nu\in\widetilde{\NN}\setminus\{0\}$ and $x\in X$.
  Note that by~\eqref{eq:entropy-and-iteration} the semigroup
  $S_k$ is closed under iterations.

  An important application of these results was given
  in~\cite{AlmeidaVolkov2}: it is easy to prove that an element $w$ of
  $\wh F_{\pv V}(X)$ belongs to the minimal ideal $K_X$ if and only if
  $h(w)=\log_2|X|$,
  hence one immediately concludes that
  $\wh F_{\pv V}(X)\setminus K_X$ is a semigroup
  closed under iterations and composition with $r$-ary implicit operations
  $w$ such that $h(w)< (\log_2|X|)^2\cdot\log_r|X|$.
    The minimal ideal $K_X$ is the $\J$-class associated to
  the full shift $X^{\mathbb Z}$. We are going to prove analogues of
  these results for sofic shifts.

  Let $F(w)$ be the set of finite factors of $w$.
  Suppose that $\ov{L(\mathscr X)}$ is factorial.
  If $w\in\ov{L(\mathscr X)}$
  then $F(w)\subseteq L(\mathscr X)$,
  and so $h(w)\leq h(\mathscr X)$.
  Note that $h(w)=h(\mathscr X)$
  if $w\in J(\mathscr X)$,
  since $F(w)=L(\mathscr X)$
  if $w\in J(\mathscr X)$. The following proposition
  gives a converse. It is an analog
  of~\cite[Corollary 4.4.9]{MarcusandLind},
  stating that if
  $\mathscr Y$ is a subshift strictly contained in a sofic shift
  $\mathscr X$ then $h(\mathscr Y)<h(\mathscr X)$.
  The proof of the proposition
  consists in a reduction to this result.

\begin{Prop}\label{p:entropy-strictly-smaller}
  Let $\pv V$ be a variety of finite semigroups containing
  $\mathbf {LSl}$.
  Suppose $\mathscr X$ is a non-periodic irreducible sofic shift
  such that $\ov{L(\mathscr X)}$ is
  a factorial subset of $\wh{F}_{\pv V}(X)$.
  If $w\in\ov{L(\mathscr X)}\setminus J(\mathscr X)$
  then $h(w)<h(\mathscr X)$.
\end{Prop}

\begin{proof}
  If $w\in X^+$ then $h(w)=0$.
  Since
  $\mathscr X$ is a non-periodic irreducible sofic shift,
  we have $h(\mathscr X)>0$~\cite[Corollary~4.4.9]{MarcusandLind},
  therefore we may suppose that
  $w\notin X^+$.

  Let $\{w_n\}$ be a sequence of elements of $L(\mathscr X)$ converging
  to $w$.
    Let $\lambda\colon X^+\to S_{\mathscr X}$
  be the syntactic morphism for $L(\mathscr X)$.
  Taking subsequences, we may suppose that
  $|w_n|\geq 2|S_{\mathscr X}|+n$
  for all $n$.
  From the proof of~\cite[Prop. 3.7.1]{Almeida:book}
  we conclude that given a homomorphism
  $\psi\colon B^+\to T$ onto a finite semigroup $T$, then,
  for every $z\in B^+$ such that
  $|z|\geq |T|$,
  there are $z_0,z_2\in B^\ast$ and $z_1\in B^+$
  such that $z=z_0z_1z_2$, $|z_0z_1|\leq |T|$ and $\psi(z_1)$ is idempotent.  Of course the dual result holds as well.
  Applying this result to
  the syntactic morphism $\lambda$,
  and to the prefix and the suffix
  of length $|S_{\mathscr X}|$ of $w_n$, we conclude
  that $w_n=w_{n,0}w_{n,1}w_{n,2}w_{n,3}w_{n,4}$
  for some
  $w_{n,i}\in X^\ast$ such that
  $|w_{n,0}w_{n,1}|\leq |S_{\mathscr X}|$ and
  $|w_{n,3}w_{n,4}|\leq |S_{\mathscr X}|$,
  $w_{n,1}$ and $w_{n,3}$ are non-empty words
  whose image under $\lambda$ is idempotent,
  and $|w_{n,2}|\geq n$.
  Taking subsequences, we may suppose that the sequence of tuples
  $\{(w_{n,0},w_{n,1},w_{n,2},w_{n,3},w_{n,4})\}$
  converges to
  $(w_{0},w_{1},w_{2},w_{3},w_{4})$.
  Then $w=w_{0}w_{1}w_{2}w_{3}w_{4}$.
  Note that
  $  (w_i)^\omega=\lim_{k\to \infty}\lim_{n\to \infty}(w_{n,i})^{k!}$.
Since $\lambda(w_{n,1})$ and $\lambda(w_{n,3})$ are idempotents
and $\lambda(w_{n,1}w_{n,2}w_{n,3})\neq 0$,
the word $(w_{n,1})^{k!}w_{n,2}(w_{n,3})^{k!}$ belongs
to  $L(\mathscr X)$ for all $k$, $n$.
Therefore, the element $v=(w_1)^\omega w_2 (w_3)^\omega$ belongs to
$\ov{L(\mathscr X)}$.

We have  $h(w)=\max\{h(w_0),h(w_1w_2w_3),h(w_4)\}$
by~\eqref{eq:entropy-and-multiplication}.
But $w_0$ and $w_4$ belong to $X^+$, because
the words $w_{n,0}$ and $w_{n,4}$ have bounded length.
Therefore $h(w_0)=h(w_4)=0$, and so
$h(w)=h(w_1w_2w_3)$.
Let $\rho$ be the ternary implicit operation
$x_1^\omega x_2x_3^\omega$, on the three-letter alphabet
$\{x_1,x_1,x_3\}$.
Note that $v=\rho(w_1,w_2,w_3)$.
Since $h(\rho)=\max\{h(x_1^\omega),h(x_2),h(x_3^\omega)\}=0$,
it follows from~\eqref{eq:entropy-and-composition}
that
\begin{equation*}
h(v)=\max\{h(w_1),h(w_2),h(w_3)\}=h(w_1w_2w_3)=h(w).
\end{equation*}
So, it suffices to prove that
  $h(v)<h(\mathscr X)$.

  By Lemma~\ref{bottomcriterion},
  the hypothesis  $w\in \ov{L(\mathscr X)}\setminus J(\mathscr X)$
  implies the existence of a word
  $u\in L(\mathscr X)$ such that $u$ is not a factor of $w$.

  The language $F(v)$ is clearly factorial,
  and it is prolongable because
  $v=w_1^\omega vw_3 ^\omega$
  belongs to $\wh{F}_{\pv V}(X)\,v\,\wh{F}_{\pv  V}(X)$.
  Therefore $F(v)=L(\mathscr Y)$ for some
  shift $\mathscr Y$.
    Since $v\in \ov{L(\mathscr X)}$
  and $\ov{L(\mathscr X)}$ is factorial,
  we know that $F(v)\subseteq L(\mathscr X)$
  and
  $F(w_1^\omega)\cup F(w_2^\omega)\subseteq L(\mathscr X)$.
  The set $F(w_1^\omega)\cup F(w_2^\omega)$
  is the language of the union of two periodic shifts,
  whence
  $F(w_1^\omega)\cup F(w_2^\omega)\neq L(\mathscr X)$,
  else $\mathscr X$ would not be irreducible non-periodic.
  Hence, there is a word $u'$ belonging to $L(\mathscr X)$
  but not to $F(w_1^\omega)\cup F(w_2^\omega)$.
  Since $\mathscr X$ is irreducible,
  there are $x,y$
  such that the word
  $u''=uxu'yu$ belongs to $L(\mathscr X)$.

  Suppose that $u''$ is a factor of $v=w_1^\omega w_2w_3^\omega$.
  The word
  $u''$ is not a factor
  of $w_1^\omega$ or $w_3^\omega$,
  (because $u'$ is a factor of $u''$),
  nor of $w_2$ (because $u$ is a factor of $u''$ and
  $w_2$ is a factor of $w$).
  By \cite[Lemma 8.2]{AlmeidaVolkov2}
  and the fact that $w_2\in \wh{F}_{\pv V}(X)\setminus X^+$,
  we have $u''=sp$ for some
  words $s$ and $p$ such that
  $s$ is a suffix of
  $w_1^\omega$ and $p$ is a prefix of $w_2$,
  or such that
  $s$ is a suffix of
  $w_2$ and $p$ is a prefix of $w_3^\omega$.
  Suppose the first case occurs.
  Then $s$ does not have $u'$ as factor, thus
  $s$ is a strict prefix of $uxu'$. But then $yu$ is a suffix of $p$,
  which is impossible, since $u$ is not a factor of $w_2$.
  The first case leads to an absurdity, and similarly so does
  the
  second case. Hence
  $u''$ is not a factor of $v$.

  Therefore $L(\mathscr Y)\subsetneq L(\mathscr X)$,
  that is, $\mathscr Y\subsetneq \mathscr X$.
  By~\cite[Corollary 4.4.9]{MarcusandLind},
  this implies
  $h(\mathscr Y)<h(\mathscr X)$.
  Then it follows trivially from
  equality $F(v)=L(\mathscr Y)$
  that $h(v)<h(\mathscr X)$.
\end{proof}

Proposition~\ref{p:entropy-strictly-smaller}
states that
$\ov{L(\mathscr X)}\setminus J(\mathscr X)$
is contained in the semigroup $S_{h(\mathscr X)}$,
stable under iterations.
In general
$\ov{L(\mathscr X)}\setminus J(\mathscr X)$ is not stable under
iterations,
but if we restrict to
endomorphisms such that
$\varphi(L(\mathscr X))\subseteq \ov{L(\mathscr X)}$
then we obtain a positive result.

Another example in which one obtains a result weaker than
in the case of the full shift,
is the following: if $I(\mathscr X)$ is the ideal generated by
$J(\mathscr X)$,
then
$S_{h(\mathscr X)}\subseteq \wh F_{\pv V}(X)\setminus I(\mathscr X)$,
but in general
$S_{h(\mathscr X)}\neq\wh F_{\pv V}(X)\setminus I(\mathscr X)$
and $\wh F_{\pv V}(X)\setminus I(\mathscr X)$
is not stable under iteration.
For example,
let $\mathscr X$ be a shift
such that $\mathsf{alph}(\mathscr X)=\{a,b\}$
and let $X=\{a,b,c\}$.
Let $u\in J(\mathscr X)$.
Consider the
endomorphisms $\psi$ and $\p$
given by
\begin{equation*}
  \psi(a)=a,\,\psi(b)=c,\,\psi(c)=b,\quad\quad
  \p(a)=a,\,\p(b)=\psi(u),\,\p(c)=b.
\end{equation*}
Then $h(\psi(u))=h(u)$
and $\p(\psi(u))=u$.
Since $\psi(u)\notin I(\mathscr X)$
and $u\in I(\mathscr X)$, it follows
that
$S_{h(\mathscr X)}\neq\wh F_{\pv V}(X)\setminus I(\mathscr X)$
and that $\wh F_{\pv V}(X)\setminus I(\mathscr X)$
is not stable under iteration.
On the other hand, $\wh F_{\pv V}(X)\setminus I(\mathscr X)$
is a semigroup whenever $\pv V=\pv A\malce \pv V$~\cite{SteinbergArxiv:2008}.

\def\malce{\mathbin{\hbox{$\bigcirc$\rlap{\kern-7.75pt\raise0,50pt\hbox{${\tt
  m}$}}}}}\def\cprime{$'$} \def\cprime{$'$} \def\cprime{$'$} \def\cprime{$'$}
  \def\cprime{$'$} \def\cprime{$'$} \def\cprime{$'$} \def\cprime{$'$}
  \def\cprime{$'$}

\end{document}